\documentclass{amsart}
\usepackage{amsmath, amsfonts, color, graphicx, multirow}
\usepackage{mathtools}
\usepackage{lineno}
\usepackage{comment, todonotes} 
\usepackage{arydshln}
\usepackage{hyperref}
\usepackage{cite}

\newtheorem{theorem}{Theorem}[section]
\newtheorem{question}[theorem]{Open Question}

\newtheorem{lemma}[theorem]{Lemma}
\newtheorem{proposition}[theorem]{Proposition}
\newtheorem{corollary}[theorem]{Corollary}

\theoremstyle{definition}
\newtheorem{definition}[theorem]{Definition}
\newtheorem{example}[theorem]{Example}

\newtheorem{remark}[theorem]{Remark}

\newcommand{\Hess}{\mathcal{H}}
\newcommand{\G}{GL(n, \mathbb C)}

\title{Which Hessenberg varieties are GKM?}
\author{Rebecca Goldin}
\author{Julianna Tymoczko}

\begin{document}

\maketitle

\begin{abstract}
Hessenberg varieties $\Hess(X,H)$ form a  class of subvarieties of the flag variety $G/B$, parameterized by an operator $X$ and certain subspaces $H$ of the Lie algebra of $G$. We identify several families of  Hessenberg varieties in type $A_{n-1}$ that are $T$-stable subvarieties of $G/B$, as well as families that are invariant under a subtorus $K$ of $T$. In particular, these varieties are candidates for the use of equivariant methods to study their geometry. Indeed, we are able to show that some of these varieties are unions of Schubert varieties, while others cannot be such unions.

Among the $T$-stable Hessenberg varieties, we identify several that are {\it GKM spaces}, meaning $T$ acts with isolated fixed points and a finite number of one-dimensional orbits, though we also show that not all Hessenberg varieties with torus actions and finitely many fixed points are GKM. 

We conclude with a series of open questions about Hessenberg varieties, both in type $A_{n-1}$ and in general Lie type. 
\end{abstract}

\section{Introduction}
\begin{center} {\it To our inspirational advisor, friend, and cheerleader, Victor Guillemin. }
\end{center}

\bigskip

GKM (Goresky-Kottwitz-MacPherson) theory is the name given to a number of algebro-combinatorial techniques to compute the equivariant cohomology of suitable spaces
with torus actions.  For a space $\mathcal M$ with a $T$-action, Goresky, Kottwitz, and MacPherson identified conditions under which the equivariant cohomology $H_T^*(\mathcal M)$ could be described in terms of the restriction of classes to 
$T$-fixed points \cite{GKM98}, also called {\it localizations}. They built on work of Atiyah and Bott \cite{AtiBot84}, who developed localization techniques in equivariant cohomology, and of Chang and Skjelbred in the 1970s, who proved that under very general conditions, the equivariant cohomology of a $T$-space could be described in terms of the cohomology of fixed point sets of proper subtori \cite{ChaSkj74}.
The key insight that sparked GKM theory was the restricted class of $T$-spaces $\mathcal M$ that Goresky, Kottwitz, and MacPherson considered, which permitted them to describe the $T$-equivariant cohomology ring of $\mathcal M$ in terms of linear relations in the dual $\mathfrak t^*$ of the Lie algebra of the torus.  

These ideas gained particular traction in the symplectic community due to the well-known convexity result of Guillemin and Sternberg \cite{GuiSte82}, and separately Atiyah \cite{Ati82}, stating that the image of the moment map for a Hamiltonian $T$-action on a compact manifold $M$ is given by the convex hull of the moment map image of the fixed point set. As this image lies in $\mathfrak t^*$ and can often be recovered (up to a translation) using localization, 
research turned to questions including: how to apply the GKM results outside the algebraic category, how to classify spaces that satisfy the (somewhat restrictive) GKM conditions of the original theory, and how to extend the theory to geometric spaces that have properties close to the original GKM requirements.

Victor Guillemin and his graduate students, postdoctoral fellows, collaborators and colleagues  developed and generalized GKM theory, popularizing it with a broad mathematical audience. Guillemin, Holm and Zara applied GKM theory to homogeneous spaces $G/H$ in \cite{GHZ06}, and then Guillemin and Zara extended it with Sabatini to certain fiber bundles with $T$ actions \cite{GuiSabZar12}.
Tolman and Weitsman \cite{TolWei99} extended the results to symplectic manifolds with Hamiltonian torus actions using Morse Theory. The first author and Holm used these results to identify the equivariant cohomology of some symplectic manifolds that are almost GKM in a specific sense \cite{GolHol01}. 
Results about the cohomology and symplectic and GIT quotients were refined using GKM theory, for example, in work by the first author, Holm and Jeffrey \cite{GHJ03}. The results also extended to orbifold (Chen-Ruan) cohomology, in both equivariant cohomology \cite{GHK07} and $K$-theory \cite{GHHK11}, and to intersection cohomology in, for example, \cite{BraPro09}. The first author and Tolman used GKM techniques as part of a strategy to generalize Schubert calculus \cite{GolTol09}. More recently, several authors have been interested in generalizing Schubert calculus to subvarieties of $G/B$, including both of the authors of this article.

Guillemin and Zara constructed a purely combinatorial model for GKM theory, extending it to {\it graphs} with additional structure that mimics the geometric properties of manifolds with group actions \cite{GuiZar00, GuiZar01}. Other combinatorial generalizations include work done by the second author on splines (see, for example,  \cite{Tym16a}).
Still others began exploring the geometry and topology of GKM spaces without regard to their cohomology rings, resulting in GKM bundles, GKM orbifolds (e.g. \cite{Zon15}), GKM sheaves (e.g. \cite{Bai14}), and GKM compatible subsets (e.g. \cite{HarTym17}).

Applications of GKM theory to cohomology theories of flag varieties were especially productive. 
The results that stemmed from GKM theory contributed to well-established literature in Schubert calculus involving combinatorics, geometry and representation theory. A seminal paper by Billey \cite{Bil99} found formulas for the localizations of all Schubert classes on flag varieties $G/B$ (see \cite{Tym16b} for a description of GKM on $G/B$). GKM theory then led to new descriptions of the equivariant cohomology and $K$-theory of $G/B$, resulting in many additions to the literature on Schubert calculus. 
This includes work by, for example, Harada, Henriques and Holm on affine Schubert varieties \cite{HHH05}, Knutson and Tao's work on the equivariant cohomology of the Grassmannian \cite{KnuTao03}, Abe and Matsumura on weighted Grassmannians \cite{AbeMat15}, and work by many others who either used GKM theory, or implicitly benefited from what became standard techniques. The introduction we provide here is necessarily incomplete.

Still now, more than two decades after the original GKM paper, GKM theory thrives in the sweet spot in which conditions remain simple and open to direct calculation, and yet result in fruitful outcomes. This may be why Victor has had almost 50 PhD students, and more than 200 ``descendants." We wrote this paper in honor of this tradition.

In this paper, we analyze torus actions on one important family of varieties called {\it Hessenberg varieties} and describe conditions that ensure certain Hessenberg varieties are GKM, namely GKM theory applies to them.  Suppose that $G$ is the group of $n \times n$ invertible matrices and $B$ is the Borel subgroup that consists of upper-triangular matrices.  Hessenberg varieties are defined by two objects: an $n \times n$ matrix $X$ and a certain linear subspace $H$ of $n \times n$ matrices satisfying Definition~\ref{definition: hess spaces}. The variety $\Hess(X,H)$ is defined by the condition
\[\Hess(X,H) = \{ \textup{ flags } gB \in G/B: g^{-1}Xg \in H\}.\]

Hessenberg varieties arise in various contexts.  In representation theory, the Hessenberg varieties $\Hess(X,\mathfrak{b})$ are called {\it Springer fibers} when $X$ is nilpotent and {\it Grothendieck-Springer fibers} for general $X$. They are used in one of the seminal constructions of a geometric representation.  T.~Springer discovered that the Weyl group acts on the cohomology of what are now called  Springer fibers \cite{Spr76, Spr78}. In Lie type $A_{n-1}$ the top-degree cohomology of $\Hess(X,\mathfrak{b})$ for nilpotent $X$ is an irreducible representation, and varying over all nilpotent conjugacy classes of $X$ recovers every irreducible representation of the Weyl group $S_n$ exactly once in that case. The Springer representation has since been constructed in many ways, using tools across geometry, topology, algebra, and combinatorics, among many others \cite{KazLus79, KazLus80, BorMac83, Hot81, GarPro92}. 

Another thrust of research asks about geometric properties of Hessenberg varieties, including: What kinds of cell decompositions do they have? \cite{Tym06a, Tym07, Pre13} What dimension are they? \cite{Ive06, Tym06a} Are they, or their components, smooth and if not, what kinds of singularities do they have? \cite{AbeIns, FreMel10, FreMel11, InsPre19, InsYon12} Which Hessenberg varieties are Schubert varieties? \cite{AbeCro16, EPS}

More recently, Hessenberg varieties emerged as an integral part of a geometric proof of the Stanley-Stembridge conjecture in combinatorial representation theory.  The Stanley-Stembridge conjecture states that a particular kind of symmetric function can be expressed as a nonnegative linear combination of the basis of elementary symmetric functions.  The symmetric functions considered by the Stanley-Stembridge conjecture are chromatic symmetric functions of the incomparability graph of a kind of poset called  $(3+1)$-free; the conjecture is related to other important claims about, e.g., immanants \cite{Ste92}.  In a sequence of papers, Shareshian and Wachs \cite{ShaWac16}, Brosnan and Chow \cite{BroCho18}, and Guay-Paquet \cite{Gua2} established an explicit relationship between the chromatic symmetric function of the Stanley-Stembridge conjecture and the cohomology of the Hessenberg varieties $\Hess(X,H)$ when $X$ is regular semisimple, under an action of the Weyl group $S_n$ that is easiest to describe using GKM theory \cite{Tym08b}.  Much work has been done recently to prove the Stanley-Stembridge conjecture using GKM-type methods for regular semisimple Hessenberg varieties \cite{AHM19, CHL, Cho}.

The basic question of this paper is: 
\begin{center}{\bf Which Hessenberg varieties are GKM}? \\ \end{center}

The answer depends on the choice of linear operator $X$, Hessenberg space $H$, as well as which torus acts.  The case when $X$ is diagonal (or more generally semisimple) has already been resolved.  When $X$ is diagonal, the variety $\Hess(X,H)$ sports an action of the torus of diagonal matrices for all $H$ and in all Lie types, and thus is GKM.  This follows from Lemma~\ref{dumbactioncondition} or, e.g., \cite{Tym08b}. When $X$ is diagonal with distinct values along the diagonal (or more generally regular semisimple), the equivariant cohomology of the variety $\Hess(X,H)$ has particularly beautiful combinatorial structure \cite{Tym08b}.  In this case, if additionally $H$ consists of $\mathfrak{b}$ together with the negative simple root spaces, then the Hessenberg variety $\Hess(X,H)$ is actually the well-known toric variety associated to the permutohedron \cite{DPS92}.  

 However, when $X$ is nilpotent, the Hessenberg variety $\Hess(X,H)$ does not usually have a ``big enough" torus action to be GKM.  For instance, suppose $X$ is regular nilpotent, namely $X$ consists of a single Jordan block, or equivalently $X = \sum_{i=1}^{n-1} E_{\alpha_i}$ where $E_{\alpha_i}$ is a nonzero element in the $\alpha_i$-weight space of $\mathfrak g$ for simple roots $\alpha_i$. The Hessenberg variety $\Hess(X,H)$ admits a one-dimensional $\mathbb C^*$-action on $\Hess(X, H)$ which is necessarily too small a torus for $\Hess(X,H)$ to be GKM under this action.  Poset pinball is one approach to analyzing the equivariant cohomology of a subvariety of a GKM variety on which the full torus does not act, by using GKM-type techniques to extrapolate information from the ambient variety \cite{HarTym17}. It has been particularly successful for analyzing nilpotent Hessenberg varieties \cite{Dre15, HarTym11, Ins15}.
 
 For one special class of nilpotent $X$, Abe and Crooks proved that the Hessenberg variety $\Hess(X,H)$ admits the full torus action, is GKM with respect to this action, and in fact is a union of Schubert varieties \cite{AbeCro16}. We give a slightly more streamlined proof that it admits a torus action and is GKM in Proposition~\ref{proposition: F1 in all types}. Balibanu and Crooks recently proposed a different direction than that of this paper, in which they classify regular Hessenberg varieties whose cohomology rings ``behave" like the cohomology of a GKM space without necessarily being GKM spaces \cite{BalCro}.

In this paper, our goal is instead to identify the largest subtorus of $T$ that acts on nilpotent Hessenberg varieties $\Hess(X,H)$.  We note that virtually nothing is known about the case when $X$ is neither semisimple nor nilpotent.  

More precisely, after establishing notation and core definitions in Section~\ref{section: background}, we prove the following results.
\begin{itemize}
    \item In Section~\ref{section: circle actions} we show that all $\Hess(X,H)$ admit the action of at least a rank-one subtorus of $T$ and identify a particular family of matrices $X$ called {\it skeletal nilpotents} that always have a $\mathbb C^*$ action with isolated fixed points.
    \item In Section~\ref{section: torus actions}, we show that, for $X$ skeletal nilpotent,  $\Hess(X,H)$ has an action of a codimension-$k$ subtorus of $T$.  We also prove that the corresponding fixed points in $\Hess(X,H)$ are isolated.
    \item  
    Theorem~\ref{theorem: main column test}  states one of our main results: a description of a system of linear equations that must hold for the orbit of an element $gB \in \Hess(X,H)$ under a particular subtorus to be in $\Hess(X,H)$. The system of equations depends on both $g$ and the parameters of the subtorus.
    \item Section~\ref{section: torus actions for Fk} then specializes to a particular family of matrices $F_k$ and uses Theorem~\ref{theorem: main column test} to identify various conditions on Hessenberg varieties that either imply or preclude an action of the full torus, in some cases generalizing the examples from the beginning of the section. In particular, Theorem~\ref{theorem: conditions for T-stability} fully characterizes the subtori of $T$ that act on $\Hess(F_2,H)$ for all Hessenberg spaces $H$.  Theorem~\ref{theorem: max subtorus on Fk} characterizes the Hessenberg spaces $H$ for which only a codimension-$k$ subtorus of $T$ acts on $\Hess(F_k,H)$.  As an application, we show that a particularly important family called {\it Peterson varieties} admits just the rank-one subtorus of $T$ used in the literature.
    \item Finally, Section~\ref{section: gkm spaces} proves that many Hessenberg varieties are GKM spaces.  This claim holds of course for all the Hessenberg varieties that admit a full torus action; see also Lemma~\ref{lemma: subvariety of GKM with same torus action is GKM} that the analogous claim is true for any subvariety of a GKM space that itself carries full torus action.  We also identify all the Hessenberg varieties $\Hess(F_2, H)$ that are GKM with respect to the codimension-one subtorus $K \subseteq T$ that acts on all Hessenberg varieties for $F_2$.  Our result demonstrates a general principle that there is a ``smallest possible Hessenberg space" under which a given torus action is GKM on the Hessenberg varieties for a particular $X$.  
    \item In particular, Remark~\ref{remark: non-Schubert GKM} establishes families of Hessenberg varieties that are GKM with respect to subtori $K \subseteq T$ but not $T$-stable, and thus are neither Schubert varieties nor unions of Schubert varieties.  It also demonstrates $T$-stable families of Hessenberg varieties that are neither unions of Schubert varieties nor homeomorphic to unions of Schubert varieties.  This answers a longstanding open question.
    \item Section~\ref{section: questions} ends with open questions.  
\end{itemize}

The first author was partially supported by National Science Foundation (NSF) grant \#2152312, and the second author was partially supported by NSF grant \#2054513.

\section{Background and notation} \label{section: background}

This section contains a quick introduction to the terminology and key objects in this paper. We begin with fundamentals about flag varieties, then review the core ideas of GKM theory and Hessenberg varieties.  We treat the case of Lie type $A_n$ though many concepts generalize to other Lie types; see also Question~\ref{question: general Lie type} in Section~\ref{section: questions}.

The flag variety is the quotient $\G/B$ where $B \subseteq \G$ is a Borel subgroup.  We take $B$ to be the upper-triangular matrices.  Alternatively, we can describe flags geometrically in terms of nested subspaces. Let $V_g^{m}$ denote the linear span of the first $m$ columns of $g\in \G$.  Then the flags in $\G/B$ can be equivalently described as the collection of sequences of nested subspaces
\[
V_g^{\bullet} = \left(V_g^1 \subseteq V_g^2 \subseteq V_g^3 \subseteq V_g^4 \subseteq \cdots \subseteq V_g^n=\mathbb C^n\right)
\]
where $V_g^i$ is $i$-dimensional.  When using explicit coordinates, we use the standard basis vectors $e_i \in \mathbb{C}^n$ that are zero in all rows except the $i^{th}$ where they are one.

Throughout we use $T$ to denote the maximal torus of $\G$ consisting of diagonal matrices. We use $W$ to denote the permutation matrices inside $\G$.  We also identify these permutation matrices both with the Weyl group and with the permutations on $\{1,2,\ldots,n\}$ according to the rule that $w(e_i) = e_{w(i)}$.

It is convenient to use the terminology of pivots from elementary linear algebra.

\begin{definition}
For any matrix, let $piv(M_j)$ denote the pivot of column $M_j$, namely the lowest nonzero row in the $j$th column vector.  If $M_j$ is the zero vector then take $piv(M_j)$ to be $0$.
\end{definition}

Gaussian elimination ensures that the pivots of a matrix are in convenient locations.  We often do calculations involving some number of columns from a square $n \times n$ matrix. For that reason, we introduce the following notation.

\begin{definition}
We say that the $n \times m$ matrix $M$ is in {\it normalized Schubert form} if there is a $n \times n$ permutation matrix $w\in W$ such that the pivots of any nonzero columns of $M$ coincide with the corresponding columns of $w$. Equivalently,  $M$ satisfies the following conditions:
\begin{itemize}
    \item The $i$th column of $M$ has a $1$ in the $w(i)$ row;
    \item The $i$th column of $M$ has $0$ in rows $\ell$ for $\ell>w(i)$; and
    \item  The $w(i)$th row of $M$ has $0$s in columns $j$ for $j>i$.
\end{itemize} 
\end{definition}
\begin{example} \label{example: schubert cells}
The $4\times 4$ matrices 
\[
M_1 = 
\begin{pmatrix}
a_{11} & a_{12} & 1 & 0\\
a_{21} & a_{22} & 0 & 1 \\
1 & 0 & 0 & 0\\
0 & 1 & 0 & 0
\end{pmatrix} \hspace{1em}
M_2 = 
\begin{pmatrix}
a_{11} & 1 & 0 & 0\\
a_{21} & 0 & 1 & 0\\
a_{31} & 0 & 0 & 1\\
1 & 0 & 0 & 0
\end{pmatrix}
\hspace{1em} 
M_3 = 
\begin{pmatrix}
a_{11} & a_{12} & a_{13} & 1\\
a_{21} & a_{22} & 1 & 0\\
a_{31} & 1 & 0 & 0\\
1 & 0 & 0 & 0
\end{pmatrix}
\]
are in normalized Schubert form, with $w_1= [3412]$, $w_2= [4123]$, and $w_3 = [4321]$ respectively.
\end{example}

One way to define Schubert cells is using normalized Schubert form, though Schubert cells are classically defined as double cosets.

\begin{definition} \label{definition: schubert cell}
For each $w\in W$, the {\it Schubert cell} $\mathcal{C}_w \subseteq \G/B$ consists of the flags $gB$ for which $g$ is in normalized Schubert form with pivots in entries $w$.  Equivalently, the Schubert cell is the coset $BwB \subseteq \G/B$. The {\it Schubert variety} is the closure of $\mathcal C_w$ in $G/B$ and is denoted $\overline{\mathcal{C}}_w$.
\end{definition}

\begin{example}
For $n=4$ and $w_1=[3412],$ the Schubert cell 
\[
\mathcal C_{w_1} = \{ \langle \vec{v}_1\rangle \subseteq \langle \vec{v}_1, \vec{v}_2\rangle \subseteq  \langle \vec{v}_1, \vec{v}_2, e_1 \rangle \subseteq \mathbb C^4\},
\]
where $\vec{v}_1= a_{11}e_1+ a_{21}e_2+ e_3$ and  $\vec{v}_2=a_{12}e_1+a_{22}e_2+e_4$ for any $a_{11}, a_{21}, a_{12}, a_{22}\in \mathbb C$. 
\end{example}

Schubert cells are important because they form a CW-decomposition of the flag variety and their closures, the Schubert varieties, induce a basis for $H^*(G/B)$.  Corollary~\ref{corollary: last row of H empty} and Remark~\ref{remark: non-Schubert GKM} discuss how Schubert cells can and cannot be used to decompose Hessenberg varieties.

\subsection{GKM theory}

GKM theory has a rich and beautiful combinatorial, geometric, and topological literature, primarily in the service of computing equivariant cohomology.  We focus instead on the question of whether GKM theory can be applied to particular varieties, namely whether a variety is a {\it GKM space}.

\begin{definition}
Let $K$ be an algebraic torus and let $\mathcal{X}$ be a complex projective algebraic variety admitting a $K$-action.  We say $\mathcal X$ is a {\it GKM space} if
\begin{itemize}
    \item The fixed point set $\mathcal X^{K}$ consists of isolated points.
    \item There are finitely many one-dimensional orbits of $K$ on $\mathcal X$.
    \item The space $\mathcal X$ is {\textit equivariantly formal} with respect to the action of $K$.
\end{itemize}
\end{definition}

Note that $\mathcal{X}$ does {\it not} need to be smooth, as the next remark details.

\begin{remark}
The definition of equivariant formality is fairly technical and has different statements (see, e.g., \cite[Theorems 1.6.2 and 14.1]{GKM98}).  In practice, we usually make additional geometric assumptions about $\mathcal{X}$ that imply equivariant formality---e.g. that $\mathcal{X}$ has no odd-degree ordinary cohomology or that $\mathcal{X}$ is a symplectic manifold with a Hamiltonian $K$-action \cite[Theorem 14.1 Parts (1) and (9)]{GKM98}.  Hessenberg varieties have no odd-degree cohomology because they have a paving by (complex) affines \cite{Tym06a} and so are equivariantly formal with respect to every torus action.
\end{remark}

With these assumptions, the boundary of each one-dimensional $K$-orbit contains two $K$-fixed points and the closure of each one-dimensional orbit is isomorphic to $\mathbb{CP}^1$ with fixed points at north and south pole.   This means the $0$- and $1$-dimensional orbits of a GKM space can be described as a graph, often called the {\it GKM graph} or {\it moment graph} of the variety.  Goresky-Kottwitz-MacPherson's theorem states that the equivariant cohomology of $\mathcal X$ is built from the moment graph, with each labeled by $K$-weights on the corresponding one-dimensional orbit, according to an algebraic algorithm. 

In some of what follows, we count orbits directly.  However, we also use the following claim, which shows that $K$-invariant subvarieties of GKM spaces are GKM themselves. The proof is straightforward but we record it because we refer to it several times.

\begin{lemma} \label{lemma: subvariety of GKM with same torus action is GKM}
If $\mathcal{X}$ is GKM with respect to a torus $K$ and $\mathcal{Y} \subseteq \mathcal{X}$ is a subvariety that is also $K$-stable then $\mathcal{Y}$ is GKM with respect to $K$. 
\end{lemma}

\begin{proof}
The $K$-fixed points $\mathcal{Y}^K$ are contained in $\mathcal{X}^K$ and the one-dimensional $K$-orbits of $\mathcal{Y}$ are a subset of the one-dimensional $K$-orbits of $\mathcal{X}$.  So both sets are finite and $\mathcal{Y}$ is GKM with respect to $K$.
\end{proof}

A variety that {\it contains} another $T$-stable variety need not itself be $T$-stable.  Indeed, the variety consisting of a single $T$-fixed point is itself $T$-stable. However, most varieties containing a $T$-fixed point are not themselves $T$-stable.  Example~\ref{example: T-stable containment} shows a particular case of this involving Hessenberg varieties.

Note that $T$ contains a one-dimensional subtorus that acts trivially on $G/B$, namely the diagonal constant matrices.  It's common to quotient by these constant matrices but using $T$ is more convenient for our algebraic calculations.  Readers accustomed to using the smaller torus will note that dimensions may be slightly different for $T$.

\subsection{Hessenberg varieties in type $A_{n-1}$}

Hessenberg varieties are subvarieties of the flag variety defined by two parameters: an element $X \in \mathfrak{g}$ and a particular kind of linear subspace $H \subseteq \mathfrak{g}$.  We give both the general definition and some concrete characterizations that apply in Lie type $A_{n-1}$.

\begin{definition} \label{definition: hess spaces}
Let $\mathfrak{g}$ be a Lie algebra with fixed Borel subalgebra $\mathfrak{b}$. A {\it Hessenberg space} $H \subseteq \mathfrak{g}$ is a linear subspace of $\mathfrak{g}$ satisfying the two conditions:
\begin{itemize}
    \item $H \supseteq \mathfrak{b}$
    \item $H$ is $\mathfrak{b}$-stable, in the sense that $H \supseteq [H,\mathfrak{b}]$.
\end{itemize}
\end{definition}

Note that $\mathfrak{b} \supseteq \mathfrak{t}$ so $\mathfrak{b}$-stability in fact implies that $H$ is a direct sum of root spaces together with $\mathfrak{t}$.  In Lie type $A_{n-1}$ this gives the following two equivalent characterizations.

\begin{definition}
In Lie type $A_{n-1}$ a {\it Hessenberg function} $h: \{1,2,\ldots,n\} \rightarrow \{1,2,\ldots,n\}$ is a map satisfying two conditions: 
\begin{itemize}
    \item $h(i) \geq i$ for all $i$
    \item $h(i) \geq h(i-1)$ for all $i=2,3,\ldots,n$.
\end{itemize}
\end{definition}

In Lie type $A_{n-1}$ there is a bijection between Hessenberg spaces $H$ and Hessenberg functions $h$ given by
\[H = \left\{M \textup{ is an } n \times n \textup{ matrix with } M_{i,j} = 0 \textup{ if } i > h(j) \right\}.\]
For this reason we use $H$ and $h$ interchangeably in what follows. 

\begin{example}\label{example: hessenberg spaces} The Hessenberg spaces
\[
H_1 = \begin{pmatrix}
* & * & * & * \\
0 & * & * & * \\
0 & 0 & * & * \\
0 & 0 & 0 & *
\end{pmatrix}
\hspace{1em}
H_2 = \begin{pmatrix}
* & * & * & * \\
0 & * & * & * \\
0 & * & * & * \\
0 & 0 & * & *
\end{pmatrix}
\hspace{1em}
H_3 = \begin{pmatrix}
* & * & * & * \\
* & * & * & * \\
* & * & * & * \\
0 & 0 & * & *
\end{pmatrix}
\]
correspond respectively to Hessenberg functions $h_1$, $h_2$, and $h_3$ where
\begin{itemize}
    \item $h_1(i)=i$ for all $i$,
    \item $h_2(1)=1, h_2(2)=3,$ and $h_2(3)=h_2(4)=4$, and 
    \item $h_3(1)=h_3(2)=3$ and $h_3(3)=h_3(4)=4.$
\end{itemize}
\end{example}

Using Hessenberg spaces (or Hessenberg functions, in type $A_{n-1}$) we define Hessenberg varieties as follows.

\begin{definition}
Fix $X \in \mathfrak{g}$ and a Hessenberg space $H \subseteq \mathfrak{g}$.  The Hessenberg variety of $X$ and $H$ is the subvariety of the flag variety defined as
\[\Hess(X,H) = \{ gB \in G/B: g^{-1}Xg \in H\}.\]
In Lie type $A_n$ this is equivalent to the linear-subspace characterization that
\[\Hess(X,H) = \{ \textup{ Flags } V_g^{\bullet}: XV_g^{i} \subseteq V_g^{h(i)} \textup{ for all } i=1,\ldots,n\}.\]
\end{definition}

Since we have fixed an ordered basis (equivalently a Borel), we can consider the intersections of each Schubert cell from Definition~\ref{definition: schubert cell} with a fixed Hessenberg variety.

 \begin{definition}
Given $\Hess(X,H)$ the {\it Hessenberg Schubert cells} are the intersections $\mathcal{C}_w \cap \Hess(X,H)$.
 \end{definition} 

For certain matrices $X$ in each conjugacy class, the Schubert cells $\mathcal{C}_w$ intersect $\Hess(X,H)$ in affine linear spaces \cite{Tym06a, Pre13}.  In a slight abuse of notation, we use the phrase  {\it Hessenberg Schubert cells} even for intersections that may not be affine linear.

We now restate conditions to determine if an element $gB\in G/B$ is in $\Hess(X, H)$. Recall that $V_g^{m}$ indicates
 the linear span of the first $m$ columns of $g$ and note that the flag $gB\in \Hess(X, H)$ if and only it the $j$th column of $Xg$ is in $V_g^{h(j)}$ for all $j=1, \dots n$. This motivates the following terminology.

\begin{definition}\label{Hessenbergconditions}
  Let $g_j$ denote the $j$th column of $g$ so $Xg_j = (Xg)_j.$ We call the conditions $Xg_j \in V_g^{h(j)}$ the {\it Hessenberg conditions.}
\end{definition}

\begin{example}\label{example: Hess schubert example}
For instance, suppose that $X$ is the $4 \times 4$ matrix whose first two columns are zero, third column is the basis vector $e_1$, and last column is $e_2$.  We can compute $X \mathcal{C}_{w_i}$ for each of the Schubert cells from Example~\ref{example: schubert cells}:
\[XM_1 = \begin{pmatrix}
1 & 0 & 0 & 0\\
0 & 1 & 0 & 0 \\
0 & 0 & 0 & 0\\
0 & 0 & 0 & 0 \\\end{pmatrix} \hspace{1em}
XM_2 = 
\begin{pmatrix}
a_{31} & 0 & 0 & 1\\
1 & 0 & 0 & 0 \\
0 & 0 & 0 & 0\\
0 & 0 & 0 & 0 
\end{pmatrix}
\hspace{1em} 
XM_3 = 
\begin{pmatrix}
a_{31} & 1 & 0 & 0\\
1 & 0 & 0 & 0 \\
0 & 0 & 0 & 0\\
0 & 0 & 0 & 0 
\end{pmatrix}
\]
Using the Hessenberg function $h_3$ from Example~\ref{example: hessenberg spaces}, the Hessenberg conditions for each of the previous Schubert cells state that the first two columns of $XM_i$ must be in the span of the first three columns of $M_i$ and the last two columns of $XM_i$ must be in $\mathbb{C}^4$.  The conditions on the last two columns are vacuous.
\end{example}

The following is an immediate consequence of the definition of Hessenberg conditions.

\begin{corollary}
Suppose that $gB \in \Hess(X,H)$ and that $Xg_j$ is not identically zero for some $j$.  Then the column $g_{j'}$ with $piv(g_{j'})=piv(Xg_j)$ is within the first $h(j)$ columns, or equivalently $j' \leq h(j)$. 
\end{corollary}

We recall the following well-known proposition that the Hessenberg variety of $X$ is homeomorphic to the Hessenberg variety of any conjugate of $X$.  Informally, this means that if we fix a basis, we may choose $X$ relative to $T$ or vice versa but gain no added generality by varying both $X$ and $T$ independently.

\begin{proposition}
Let $X \in \mathfrak{g}$, $g \in G$, and $H \subseteq \mathfrak{g}$ be a Hessenberg space.  Then the map sending $g'B \in \Hess(X,H)$ to $gg'B  \in \Hess(gXg^{-1}, H)$ is a homeomorphism
\[\Hess(X,H) \cong \Hess(gXg^{-1}, H).\]
\end{proposition}

To streamline calculations in later sections, we also note that acting on the Hessenberg variety by ${\bf t} \in T$ is equivalent to conjugating $X$ by ${\bf t}^{-1}$.  

\begin{lemma}\label{dumbactioncondition} For any ${\bf t}\in T$ the flag ${\bf t}gB\in \Hess(X,H)$ if and only if $gB\in \Hess({\bf t}^{-1}X{\bf t},H).$
\end{lemma}
\begin{proof}
$gB\in \Hess({\bf t}^{-1}X{\bf t},H)$ if and only if $g^{-1}({\bf t}^{-1}X{\bf t})g\in H$ namely $({\bf t}g)^{-1}X({\bf t}g)\in H$ or equivalently ${\bf t}gB\in \Hess(X,H)$.
\end{proof}

The following definition gives a class of nilpotent operators called {\it skeletal nilpotents} that is useful for our calculations.  All nilpotent matrices in Jordan canonical form are skeletal, so every nilpotent operator is conjugate to a skeletal nilpotent matrix.  Furthermore, conjugating a skeletal nilpotent matrix by an element of $T$ gives another skeletal nilpotent matrix.
Note that any strictly upper-triangular matrix is in fact nilpotent.

\begin{definition}\label{def:Xrules}
We say that a nilpotent operator $X$ is {\it skeletal nilpotent} when
\begin{itemize}
    \item $X$ is strictly upper-triangular, and
    \item $X$ has at most one nonzero entry in each column and in each row.
\end{itemize}
If $X$ has a nonzero $i^{th}$ row, we write $X(i)$ to indicate the (unique) column such that the entry of $X$ in position $(i, X(i))$ is nonzero. If the $i^{th}$ row of $X$ is zero then we assign $X(i)=0$.
\end{definition}
\begin{example}
The matrix
\[
X = 
\begin{pmatrix}
0 & 3 & 0 & 0\\
0 & 0 & 0 & -2\\
0 & 0 & 0 & 0\\
 0& 0 & 0 & 0
\end{pmatrix}
\]
is skeletal nilpotent. 
In this case, $X(1)=2, X(2)=4,$ and $X(3)=X(4)=0.$ 
\end{example}

The skeletal nilpotent condition ensures that each nonzero row $i$ of $Xg$ is row $i'$ of $g$ scaled by the unique entry in position $(i,i')$ in $X$.  Furthermore since $X$ is strictly upper-triangular, each row of the matrix $Xg$ is either zero or a scalar multiple of a row below it in $g$.  

The next definition describes a family of skeletal nilpotent operators that we analyze extensively in what follows.

\begin{definition}
Fix $k \in \{1,\ldots,n-1\}$.  Define $F_k$ to be the matrix that is one in entries $(i,n-k+i)$ for $i=1,\ldots,k$, and zero otherwise.
\end{definition}

\begin{example} When $n=4$ the three $F_k$ are:
\[
F_{3}=
\begin{pmatrix}
0 & 1 & 0 & 0 \\
0 & 0 & 1 & 0 \\
0 & 0 & 0 & 1 \\
0 & 0 & 0 & 0
\end{pmatrix}
\hspace{1em}
F_{2}=
\begin{pmatrix}
0 & 0 & 1 & 0 \\
0 & 0 & 0 & 1 \\
0 & 0 & 0 & 0 \\
0 & 0 & 0 & 0
\end{pmatrix}
\hspace{1em}F_{1}=
\begin{pmatrix}
0 & 0 & 0 & 1 \\
0 & 0 & 0 & 0 \\
0 & 0 & 0 & 0 \\
0 & 0 & 0 & 0
\end{pmatrix}
\] 
In general $F_{n-1}$ is the only regular nilpotent matrix of the form $F_k$ for some $k$.  The matrices $F_{n-1}$ and $F_1$ are also the only matrices of the form $F_k$ with exactly one nonzero Jordan block.
\end{example}
Observe that $(F_k)^{k+1}$ is identically zero. The case $F_1$ was studied extensively by Abe and Crooks \cite{AbeCro16}.

\begin{example} \label{example: T-stable containment}
Recall the comment after Lemma~\ref{lemma: subvariety of GKM with same torus action is GKM} that a variety {\it containing} another $T$-stable variety need not be $T$-stable.  This is still true even with the additional structure of Hessenberg varieties.  Indeed, take $H=\mathfrak{b}$ and take $X=F_{n-1}$ to be a single Jordan block.  Then the Hessenberg variety $\Hess(F_{n-1},\mathfrak{b})$ is just the identity flag, which is a $T$-fixed point.  Every Hessenberg variety $\Hess(X',H')$ for which $X'$ is upper triangular contains $\Hess(F_{n-1},\mathfrak{b})$ but we show in later results (see, e.g., Corollaries~\ref{corollary:B nonzero} and~\ref{corollary: first column condition for Fk} and Theorems~\ref{theorem: conditions for T-stability} and~\ref{theorem: max subtorus on Fk}) that many Hessenberg varieties do not admit a full torus action. 
\end{example}

\section{$\mathbb C^*$-actions on $\Hess(X,H)$} \label{section: circle actions}

The question of whether a Hessenberg variety carries the structure of a GKM space relative to a subgroup of $T$ relies on it having a large enough torus action so that fixed points are isolated. In this section we just focus on the condition that a torus action produces isolated fixed points and study only rank-one subtori. We refer to the action of a one-dimensional complex torus isomorphic to $\mathbb C^*$ as a {\em $\mathbb C^*$-action}.  In this section we describe a class of Hessenberg varieties that have a $\mathbb C^*$-action with isolated fixed points, establishing a class of varieties that could potentially be GKM with respect to some (larger) subtorus of $T$. 

Consider the one-parameter subgroup of $T$ defined as follows: 
$$
S := \left\{\begin{pmatrix} t& 0 & \cdots & 0\\
0 & t^2 & \cdots & 0\\
\vdots\\
0 & 0 &\cdots & t^n
\end{pmatrix}\!,\ t\in \mathbb C^*\right\}.
$$  Note that $S$ is regular, and hence $(G/B)^S = (G/B)^T.$

The following two lemmas show that $\Hess(F_k,H)$ is $S$-stable and that in fact every skeletal nilpotent Hessenberg variety carries the action of a conjugate of $S$.

\begin{lemma}\label{lemma:Sstable} The variety $\Hess(F_k,H)$ is $S$-stable for all $k$ and $H\subset \mathfrak g$.
\end{lemma}
\begin{proof}
Note that $S$ projectively stabilizes the nilpotent $X$. More specifically, for any matrix ${\bf s}\in S$ we have ${\bf s}^{-1}F_k{\bf s} = t^{n-k}F_k.$ By Lemma~\ref{dumbactioncondition} it follows that $S$ acts on $\Hess(F_k, H)$.
\end{proof}

\begin{lemma} \label{lemma: circle group for skeletal}
Suppose that $X$ is skeletal nilpotent, as defined in Definition~\ref{def:Xrules}. Then $\Hess(X,H)$ has a $\mathbb C^*$-action by $wSw^{-1}\subseteq T$ for some permutation matrix $w$.
\end{lemma}
\begin{proof}

We construct a permutation $w$ as a product of disjoint cycles to establish that the lemma holds. 

Let $\mathcal R = \{i_1 < \dots < i_k\}$ denote the set of nonzero rows of $X$ indexed in increasing order and let $\mathcal C= \{X(i_1), \dots, X(i_k)\}$ denote the corresponding set of (nonzero) columns.  Since $X$ is skeletal nilpotent, the set $\mathcal C$ consists of $k$ distinct values, possibly including some values in $\mathcal R$.

Our notational conventions mean $i_1$ is the first row in which $X$ is nonzero.  Define $w(i_1)=1$.  Since $i_1< X(i_1)$ we may choose $w(X(i_1))=w(i_1)+1=2.$ If $X(i_1)\in \mathcal R$ also indexes one of the nonzero rows of $X$,
define $w(X^2(i_1))=w(X(i_1))+1=3.$ Each time $X^j(i_1)\in \mathcal R$ we continue to define 
\[w(X^{j}(i_1)) = w(X^{j-1}(i_1))+1.\] 
Note that there are no repeated values since $i_1\neq X(i_1)$ and all the rows 
\[1, X(i_1), X^2(i_1), \ldots \]
are distinct. We stop when $X^{j_1}(i_1)\not\in \mathcal R$ for some $j_1$.

We start the next cycle with the smallest element $i_{j_2} \in \mathcal{R}$ that does not yet have a value assigned to it.  Let $w(i_{j_2}) = w(i_{j_1})+1$ and as before assign \[w(X^j(i_{j_2}) = w(X^{j-1}(i_{j_2}))+1\] 
for all $j$ with $X^j(i_{j_2}) \in \mathcal{R}$.

We continue until all elements of $\mathcal R$ and $\mathcal C$ have assigned values and then extend $w$ to the rest of $\{1,2,\ldots,n\}$ arbitrarily, e.g. so that there are no inversions within the remaining input and output. The permutation $w$ has the property that for all $i\in \mathcal R$, we have $w(X(i)) = w(i)+1$.

Observe that $wSw^{-1}$ consists of matrices ${\bf s}$ with diagonal entries 
$$(t^{w(1)}, t^{w(2)}, \ldots, t^{w(n)}).$$
The $i$th row of 
${\bf s}^{-1} X {\bf s}$  consists of  entries in the $i$th row of $X$ multiplied by $$t^{-w(i)}t^{w(X(i))} = t,$$ since $w(X(i)) = w(i)+1$. Therefore  ${\bf s}^{-1} X {\bf s}=tX$ for some $t\in \mathbb C^*$.
 
By Lemma~\ref{dumbactioncondition}, it follows that $wSw^{-1}$ acts on $\Hess(X,H)$. 
\end{proof}
The reader may already be identifying circumstances in which more than one $w$ makes the lemma true. These may be circumstances under which $\Hess(X,H)$ admits a larger torus action. However, the torus generated by $w^{-1}Sw$ and $w'^{-1}Sw'$ may be the same as that generated by $w^{-1}Sw$ and trivially acting $\mathbb C^*$ (represented by constant diagonal matrices).

The torus $wSw^{-1}$ is a regular subgroup of $T$ as well, so it acts with isolated fixed points. More formally we have the following.

\begin{corollary}\label{corollary: circle isolated fixed points}
Suppose $X$ is skeletal nilpotent and $wSw^{-1}$ is a one-dimensional group given by Lemma~\ref{lemma: circle group for skeletal}. Then $wSw^{-1}$ acts on $\Hess(X,H)$ with isolated fixed points. 
\end{corollary}
\begin{proof}
Observe that that $wSw^{-1}$ is a regular subgroup of $T$.  Since the fixed points $(G/B)^{wSw^{-1}}$ are finite, so are the points in $(\Hess(X,H))^{wSw^{-1}}$.
\end{proof}

\section{Torus actions on Hessenberg varieties} \label{section: torus actions}

A number of nilpotent Hessenberg varieties are $T$-stable proper subsets of $G/B$ and therefore manifestly GKM spaces.  However, most nilpotent Hessenberg varieties are not $T$-stable. They may nonetheless be GKM spaces with respect to some smaller torus action.

In this section, we describe conditions under which $\Hess(X,H)$ is invariant under a $K$ action for a subtorus $K\subseteq T$ that we identify explicitly.  (In some cases $K$ is $T$ itself.)  The main result is Theorem~\ref{theorem: main column test}, which for each $gB \in \Hess(X,H)$ produces a system of linear equations that the entries of $g$ and of the torus $K$ must satisfy for $K$ to act on $\Hess(X,H)$.

The section is organized as follows.  Section~\ref{section: torus actions on x} begins with some results about torus actions on $\Hess(X,H)$ for specific families of $X$.  Section~\ref{section: torus actions on h} then describes torus actions on $\Hess(X,H)$ for specific $H$.  The strategies and arguments in these special cases help motivate the more general arguments in Section~\ref{section: torus actions on x and h}, which address torus actions on $\Hess(X,H)$ when both $X$ and $H$ are allowed to vary. 

\subsection{Torus actions on $\Hess(X,H)$ for $X$ skeletal nilpotent}\label{section: torus actions on x}

Suppose that $X$ is skeletal nilpotent. We establish conditions for $\Hess(X,H)$  to have a torus action without restricting the subspace $H$.

The idea of the next theorem is that conjugating $F_k$ imposes $k-1$ conditions, which can be observed by direct computation.  Our proof profits from the method of Lemma~\ref{lemma:Sstable}, with conditions that the nonzero entries after conjugation must all equal.  As we shall see, these conditions impose $k-1$  relations on $T$.

\begin{theorem}\label{thm:X action by codim k-1 torus} Let $X$ be skeletal nilpotent 
with $k$ nonzero rows for some $k\leq n$.  
For any subspace $H\subseteq \mathfrak g$ the Hessenberg variety $\Hess(X,H)$ is invariant under a codimension-$(k-1)$ subtorus $K$ of $T$.

Moreover the subtorus $K$ is defined by the rule $(t_1,t_2, \dots, t_n)\in K$ if and only if 
\begin{equation}\label{torusconstraintX}
    t_{i_1}^{-1}t_{X(i_1)} =t_{i_2}^{-1}t_{X(i_2)}=\cdots=t_{i_k}^{-1}t_{X(i_k)},
\end{equation}
where $\{i_1, i_2, \ldots, i_k\}$ is the set of nonzero rows of $X$.
\end{theorem}
\begin{proof}
Let ${\bf t}=(t_1, \dots, t_n)\in T$ be a generic element indicated by its diagonal matrix entries. We note that ${\bf t}^{-1}X{\bf t}$ has nonzero entries only in the positions that $X$ does. If the entry of $X$ in position $(i, X(i))$ is $1$ then the  entry of ${\bf t}^{-1}X{\bf t}$ in position $(i, X(i))$ is $t_i^{-1}t_{X(i)}$. 

Impose the condition that all these entries have the same value $c$.  This gives
$${\bf t}^{-1}X{\bf t}=cX.$$
Thus for each subspace $H$ we have $({\bf t}g)^{-1}X({\bf t}g)=c g^{-1}Xg \in H$ if and only if $g^{-1}Xg\in H$. It follows that for each subspace $H$ we have $gB\in \Hess(X,H)$ if and only if ${\bf t}gB\in \Hess(X,H)$. By hypothesis, there are at most $k$ nonzero rows of $X$ and so there are
 up to $k-1$ relations given by \eqref{torusconstraintX}.  In the maximal case, there is a codimension-$k-1$ torus $K\subset T$ of elements that satisfy the conditions of \eqref{torusconstraintX}.
\end{proof}

When we take $X=F_k$  in Theorem~\ref{thm:X action by codim k-1 torus}, we obtain the following corollary.

\begin{corollary}
For all $H\subseteq \mathfrak g$, the variety $\Hess(F_k,H)$ is $K$-stable for the codimension $k-1$ torus $K\subseteq T$ given by 
\[\frac{t_{1+(n-k)}}{t_1} = \frac{t_{2+(n-k)}}{t_2} = \cdots = \frac{t_n}{t_k}.\]
In particular, the variety $\Hess(F_1, H)$ is $T$-stable.
\end{corollary}

We have not shown that a larger torus does {\it not} act, and there are obvious circumstances in which it does. For example, if $H=\mathfrak g$ then the full torus $T$ acts on $\Hess(X, H)=G/B$ for all $X$.  A larger torus could also act if the condition $gB \in \Hess(X,H)$ implied some relations in Equation~\eqref{torusconstraintX} never arise, for instance if all $gB \in \Hess(X,H)$ are zero in entries that would otherwise be scaled by $t_i^{-1}t_{X(i)}$.

In the next section, we examine how varying the subspace $H$  can produce a larger class of $T$-stable Hessenberg varieties, all of which are GKM spaces.

\subsection{Torus actions on $\Hess(F_k,H)$ as $H$ varies} \label{section: torus actions on h}

The case when $X=F_k$ provides a model for the more general skeletal nilpotent case. When we restrict $H$ to specific shapes, we can often identify the largest torus subgroup of $T$ under which $\Hess(F_k,H)$ is invariant.

\begin{theorem}\label{thm:Hn-1} 
Consider the variety $\Hess(F_k, H)$ with $h(j)=n-1$ for all $j<n$ and $h(n)=n$. Then $\Hess(F_k, H)$ is a $T$-stable subvariety of $G/B$.
\end{theorem}
\begin{proof}
We know $g\in \Hess(F_k, H)$ if and only if $g$ satisfies the Hessenberg conditions, i.e. $F_kg_j\in V_g^{h(j)}$ for all columns $j$. By the assumptions on $h$, we know $F_kg_j\in V_g^{n-1}$ for $j\neq n$. Note that $F_kg_j=0$ unless $j =n-k+1, \dots, n$. 

The matrix $X={\bf t}^{-1}F_k{\bf t}$ has the same nonzero columns as $F_k$. 
We need to show that $Xg_j\in V_g^{n-1}$for $j<n$.  If $piv(g_n) = i$ for some $i=1, \dots, k$, then there is some $g_j \neq g_n$ such that $piv(Xg_j)=i$. But then $j<n$ and hence $h(j)=n-1$. By the Hessenberg condition $Xg_j \in V_g^{n-1}$.  However the only column of the matrix $g$ with a pivot in the $i$th row is the last column, resulting in a contradiction.

If $piv(g_n)=i$ for $i>k$ then the columns of $g$ with pivots in positions $1, 2,\dots, k$ must be in $V_g^{n-1}$ and hence $\langle e_1, \dots, e_k\rangle \subseteq V_g^{n-1}$. Since $Xg_j\in \langle e_1, \dots, e_k\rangle$ we obtain $Xg_j\in V_g^{n-1}$ for $j<n$. When $j=n$, the Hessenberg conditions are trivially satisfied.  
\end{proof}

\begin{remark}\label{rmk: proper subsets of G/B}
Though the Hessenberg variety $\Hess(F_k, H)$ with $h(j)=n-1$ for all $j< n$ and $h(n)=n$ is GKM, it is not the entire flag variety. In particular, if the last column of $g$ has a pivot in row $1, 2,\dots, k$, then $gB\not \in \Hess(X, H)$.  We  established that the vectors $e_1, \dots, e_k$ are in the span of the first $n-1$ columns of $g$ so the last column's pivot cannot be in the first $k$ rows. 
\end{remark}

Theorem~\ref{thm:Hn-1} does not hold when $H$ is slightly larger.  We will prove in Theorem~\ref{theorem: conditions for T-stability} that $\Hess(F_2,H)$ does not admit the full torus action for any $H$ that have an $m$ with $1<m<n$ so that $h(m)=n$ and $h(\ell)=n-1$ for $\ell<m$.  

In Theorem~\ref{theorem: main column test}, we generalize these ideas to arbitrary $X$ and $H$.

\begin{remark} Let $\Hess(F_k, H)$ be the Hessenberg in which $h(j)=n-1$ for $j<n$ and $h(n)=n$. Then we can see the Hessenberg variety $\Hess(F_k, H)$ is in fact a union of Schubert varieties. Suppose $gB\in \Hess(F_k,H)$. The Schubert cell containing $g$ consists of those $g'B$ which have the same pivots as $g$. 
If $g'$ has the same pivots as $g$ then by the same argument as in Theorem~\ref{thm:Hn-1} each vector $Xg'_j$ is in $V_g^{n-1}$ for $j<n$ and thus satisfies the Hessenberg conditions. It follows that each $\Hess(F_k, H)$ is a union of Schubert cells. Since $\Hess(F_k,H)$ is closed it contains the closure of these Schubert cells, and hence is itself a union of Schubert varieties.
\end{remark}

\subsection{Torus actions on $\Hess(X,H)$ as $X$ and $H$ vary} \label{section: torus actions on x and h}

Motivated by the kinds of arguments that appear in the special cases that precede this section, we now generalize both $X$ and $H$. Our main theorem identifies constraints on the entries of $g$ and ${\bf t}$ when tori act on Hessenberg varieties.  We begin with a definition to establish notational conventions.

\begin{definition}
Let $M$ be an $n \times m$ matrix each of whose columns has a pivot, with nonzero entries only in positions both to the left and above a pivot.  Let $\Vec{w}$ be a nonzero $n \times 1$ column vector with pivot in row $\ell$.  

Suppose $M$ has exactly $k$ columns with pivots in rows $1, 2, \ldots, \ell$.  Denote the index set of these columns by $\mathcal{C}$ and the index set of their pivot rows by $\mathcal{R}\subseteq \{1, 2,\ldots, \ell\}$. Then 
\begin{itemize}
    \item the $k \times k$ {\it pivot matrix} of the system $M|\Vec{w}$ is
    \[A = \left( m_{ij}: i \in \mathcal{R}, j \in \mathcal{C} \right), \]
    \item the $(\ell-k) \times k$ {\it dependent matrix}  of the system $M|\Vec{w}$ is
    \[B = \left( m_{ij}: i \not \in \mathcal{R} \textup{ and also } i \leq \ell, j \in \mathcal{C} \right),\]
    \item the $k \times 1$ {\it solution vector}  of the system $M|\Vec{w}$ is
    \[\Vec{v} = ( w_i )_{i \in \mathcal{R}}, \mbox{ and}\]
    \item the $(\ell - k) \times 1$ {\it constraint vector} of the system $M|\Vec{w}$ is
    \[\Vec{v}' = ( w_i )_{i \not \in \mathcal{R}}.\]
\end{itemize}

We often refer to {\it the pivot matrix} without specifying the system if it is clear from context, and similarly for the others. 
\end{definition}

The following lemma establishes basic linear algebra relations between the pivot matrix, dependent matrix, solution vector, and constraint vector.

\begin{lemma}\label{lemma:pivot dependent solution constraint}
The linear system $M | \Vec{w}$ can be transformed into an equivalent system
\[
\left(\begin{array}{c|c||c}
{\Large{A}}
& \begin{array}{c} \\ \makebox[0.4in]{\Large{*}} \\ \hspace{1em} \end{array} & \Vec{v} \\
\cdashline{1-3} \begin{array}{c} \\ 
\makebox[0.4in]{\Large{B}} \\ 
\hspace{1em} \end{array} & {\Large{*}} & \Vec{v}' \\
\cline{1-3} {\Huge{0}} & 
\begin{array}{c} \\
 \makebox[0.4in]{\Large{C}} \\ 
 \hspace{1em} \end{array} &  \Vec{0}^T
\end{array} \right)
\]
by a sequence of column transpositions and row transpositions (equivalently right-multiplication by one permutation matrix and left-multiplication by another).    

The pivot matrix $A$ is invertible.  If the system $M|\Vec{w}$ has a solution then the solution is unique.  In this case, the unique solution to the equivalent system is $A^{-1} \Vec{v}$ followed by $m-k$ zeros, and the constraint vector satisfies
\[
BA^{-1}\Vec{v} = \Vec{v}'.
\]
\end{lemma}

\begin{proof}
By construction, we can cyclically permute the columns of $M|\Vec{w}$ one at a time so that those in positions $\mathcal{C}$ move to the first $k$ columns, and similarly for the rows.  These operations are equivalent to right-multiplication by an $m \times m$ permutation matrix and left-multiplication by an $n \times n$ permutation matrix.  (In fact, the $n \times n$ row permutation fixed rows $\ell+1, \ell+2, \ldots, n$ but we do not need this detail.  The last column of the system in the statement of the Lemma is the output when this row permutation is applied to $\Vec{w}$.)

Also by construction, the matrix $A$ has one pivot in each row and one pivot in each column, so it is invertible.  Since $M$ has one pivot in each column, the last $m-k$ columns also have exactly one pivot each.  Furthermore, they are all in the submatrix labeled $C$ since otherwise their pivots would be in the first $\ell$ rows.  Thus the system has at most one solution for any $\Vec{w}$.

Finally, if there is a solution to this system then it must be $A^{-1} \Vec{v}$ because none of the last $m-k$ columns can contribute without creating a pivot in a row greater than $\ell$.  Thus $BA^{-1}\Vec{v} = \Vec{v}'$ which proves the claim.
\end{proof}

Observe that, for $g$ in normalized Schubert form, $Xg_j$ satisfies the Hessenberg conditions of Definition~\ref{Hessenbergconditions} if and only if there is a solution to the linear system  
\begin{equation}\label{eq: Hessenberg conditions}
    \langle g_1, g_2, \dots, g_{h(j)}| Xg_j\rangle.
\end{equation}
We denote the pivot matrix in the corresponding equivalent system 
by $A_j$, and the dependent matrix of the system by $B_j$. Denote the pivot row of $Xg_j$ by $\ell_j$. Then $gB\in \Hess(X,H)$ if and only if \eqref{eq: Hessenberg conditions} has a solution for each $j=1, \dots n$,  which occurs if and only if 
\begin{equation}\label{eq: matrix Hessenberg conditions}
    B_j A_j^{-1}\vec v_j = \vec v'_j
\end{equation}
for all $j$, where $\vec v_j$ is the solution vector and $\vec v'_j$  is the constraint vector obtained from entries of $Xg_j$.

\begin{example}\label{example: matrices A and B in Hessenberg conditions}
Let $h=(3,3,4,4)$, and $X=F_2$. Consider the linear system $\langle g_1, g_2, \dots, g_{h(j)}| Xg_j\rangle$
for $j=1$,
for each Schubert cell of Example~\ref{example: schubert cells}.
The final column is $Xg_1$ and the matrix $C$ consists of the boxed entries in the bottom left in each case.
\[
\left\langle \begin{array}{ccc|c} a_{11} & a_{12} & 1 & 1 \\ \cline{1-2} a_{21} & \multicolumn{1}{c|}{a_{22}} & 0 & 0 \\ 1 & \multicolumn{1}{c|}{0} & 0 & 0 \\ 0 & \multicolumn{1}{c|}{1} & 0 & 0 \end{array} \right\rangle \hspace{1em}
\left\langle \begin{array}{ccc|c} a_{11} & 1 & 0 & a_{31} \\ a_{21} & 0 & 1 & 1 \\ \cline{1-1} \multicolumn{1}{c|}{a_{31}} & 0 & 0 & 0 \\ \multicolumn{1}{c|}{1} & 0 & 0 & 0 \end{array} \right\rangle \hspace{1em}
\left\langle \begin{array}{ccc|c} a_{11} & a_{12} & a_{13} & a_{31} \\ a_{21} & a_{22} & 1 & 1 \\ \cline{1-2} a_{31} & \multicolumn{1}{c|}{1} & 0 & 0 \\ 1 & \multicolumn{1}{c|}{0} & 0 & 0 \\ \end{array} \right\rangle
\]
In the first case, $A$ is the $1 \times 1$ identity matrix and $B$ is the empty matrix, meaning the $0 \times 0$ matrix with no entries.  In the second case, $A$ is the $2 \times 2$ identity matrix and---again---the matrix $B$ is the empty matrix. In the third case, $A$ is the $1 \times 1$ identity matrix, $B$ is the $1 \times 1$ matrix $a_{13}$.

By contrast, note that the third column $Xg_3$ in each case is zero.  This means that both $A$ and $B$ are empty matrices.
\end{example}

The next theorem is the core of the main results on torus actions on Hessenberg varieties that follow.  
We use the linear system described in \eqref{eq: matrix Hessenberg conditions} associated to elements $({\bf t}^{-1}X{\bf t})g$ for ${\bf t}\in T$ to produce conditions under which $gB\in \Hess(X,H)$ implies ${\bf t} gB\in \Hess(X,H)$.

\begin{theorem} \label{theorem: main column test}
Suppose that $g$ is a matrix in normalized Schubert form and $X$ is skeletal nilpotent.
Suppose $gB$ is in $\Hess(X,H)$ and that ${\bf t} \in T$.  Then ${\bf t}gB \in \Hess(X,H)$ only if the entries of $g$ and of ${\bf t}$ satisfy the following conditions: For each $j$, 
\begin{equation} \label{equation: conditions on entries}
 B_jA_j^{-1} {\bf s}_j \Vec{v}_j  =  {\bf s'}_j B_jA_j^{-1} \Vec{v}_j 
\end{equation}
where 
\begin{itemize}
    \item the vector $\Vec{v}_j$ is the solution vector and the matrices $A_j$ and $B_j$ are the pivot and dependent matrix of the system $\left( g_1 g_2 \cdots g_{h(j)} | Xg_j \right)$ obtained from the first $h(j)$ columns of $g$ together with the image $Xg_j$, and
    \item the tori ${\bf s}_j = \left(\frac{t_{X(i)}}{t_i}: i \in \mathcal{R}_j\right)$ and ${\bf s'}_j = \left(\frac{t_{X(i)}}{t_i}: i \not \in \mathcal{R}_j\right)$ for $X(i)$ the unique column in row $i$ in which $X$ is nonzero, and $\mathcal R_j$ the set of pivot rows in the first $piv(Xg_j)$ rows of $\left( g_1 g_2 \cdots g_{h(j)} | Xg_j \right)$.
\end{itemize}
\end{theorem}

\begin{proof}
Begin by considering the $j^{th}$ column $g_j$ of $g$. The relation $Xg \in gH$ holds if and only if $Xg_j$ is in the span of $\{g_1, \ldots, g_{h(j)}\}$ for each $j$.  

To find {\it how} to write $Xg_j$ as a span of the first $h(j)$ columns of $g$ we solve the augmented matrix $\left( g_1 g_2 \cdots g_{h(j)} | Xg_j \right)$.  Let $A_j$ and $B_j$ denote the pivot and dependent matrix respectively of the system $\left( g_1 g_2 \cdots g_{h(j)} | Xg_j \right)$ and let $\Vec{v}_j$ and $\Vec{v}'_j$ denote its solution and constraint vectors respectively.
Lemma~\ref{lemma:pivot dependent solution constraint} gives the following relation on the entries of the first $h(j)$ columns of $g$:
\begin{equation} \label{equation: solving Hess condition}
 B_jA_j^{-1} \Vec{v}_j = \Vec{v}_j'.
\end{equation}

Now suppose we replace $X$ with ${\bf t}^{-1}X{\bf t}$.  The only part that changes is the last column of the matrix, in which $Xg_j$ is replaced by ${\bf t}^{-1}X{\bf t}g_j$.  Since ${\bf t}^{-1}X{\bf t}$ simply rescales each entry of $X$, the vector $({\bf t}^{-1}X{\bf t})g_j$ just rescales each entry of $Xg_j$.  If the entry in the $i^{th}$ row of $X$ is $1$, then the entry in the $i^{th}$ row of ${\bf t}^{-1}X{\bf t}$ is $\frac{t_{X(i)}}{t_i}$.  Consequently if the $i^{th}$ entry of $Xg_j$ is $g_{X(i),j}$ then the $i^{th}$ entry of ${\bf t}^{-1}X{\bf t}g_j$ is $\frac{t_{X(i)}}{t_i} g_{X(i),j}$.  

Let ${\bf t}$ be the diagonal matrix with entries $(t_1, \dots, t_n)$.  Our convention is that $t_0 =1$.  Define the tori
\[{\bf s}_j = \left(\frac{t_{X(i)}}{t_i}: i \in \mathcal{R}_j\right) \]
and 
\[{\bf s'}_j = \left(\frac{t_{X(i)}}{t_i}: i \not \in \mathcal{R}_j\right). \]
By construction the solution vector of the system $g_1 g_2 \cdots g_{h(j)}|\left( {\bf t}^{-1}X{\bf t}\right)g_j$ is ${\bf s}_j \Vec{v}_j$ while the constraint vector is ${\bf s'}_j \Vec{v}_j'$.  The expression $t_0$ appears in the $i^{th}$ parameter of ${\bf s}_j$ or ${\bf s'}_j$ if and only if $X(i)=0$.  This in turn is equivalent to saying that the $i^{th}$ row of $X$ is zero and so the corresponding entry of $\Vec{v}_j$, respectively $\Vec{v}_j'$, is zero.  So in fact only expressions of the form $t_{i'}/t_i$ for coordinates of $T$ appear in nonzero terms of the product ${\bf s}_j \Vec{v}_j$ respectively ${\bf s'}_j \Vec{v}_j'$.
 
Thus we get an analogue of Equation~\eqref{equation: solving Hess condition}
\begin{equation} \label{equation: solving torus Hess condition}
 B_jA_j^{-1} {\bf s}_j \Vec{v}_j = {\bf s'}_j \Vec{v}_j'.
\end{equation}

Combining Equations~\eqref{equation: solving Hess condition} and~\eqref{equation: solving torus Hess condition}, we obtain
\[
 B_jA_j^{-1} {\bf s}_j \Vec{v}_j  =  {\bf s'}_j B_jA_j^{-1} \Vec{v}_j, 
\]
where the entries in $B_j, A_j, \Vec{v}_j, \Vec{v}_j'$ are all determined uniquely by $g$ and the entries in ${\bf s}_j, {\bf s'}_j$ are all determined uniquely by $t$ and the nonzero entries in $X$.  Varying over the $j$ gives one system for each column $j$ of $g$ as claimed.
\end{proof}

\begin{remark}
Note that the decompositions of $A_j$ and $B_j$ are determined by the pivot positions in $g$ and are independent of the choice of $g$ within its fixed Hessenberg Schubert cell (though the particular entries in $A_j, B_j$ are determined by $g$ itself).
\end{remark}

We give a corollary that starts with a simpler and immediate consequence of the Hessenberg conditions, and then analyzes the case when $A_j$ consists of a single column.

\begin{corollary} \label{corollary: one column pivot conditions}
We use the notation of Theorem~\ref{theorem: main column test}.  Suppose $X$ is skeletal nilpotent and there exists an element $gB\in \Hess(X, H)$ for which $A_j$ is a $1\times 1$ matrix for some column $g_j$ of $g$. Let $\ell = piv(Xg_j)$.  Then the elements $(t_1, \ldots, t_n)$ of any torus that acts on $\Hess(X,H)$ must satisfy the equations 
\begin{equation}
    \frac{t_{X(i)}}{t_i}= \frac{t_{X(\ell)}}{t_\ell} 
\end{equation}
for all rows $i$ in which $Xg_j$ is nonzero.
\end{corollary}

\begin{proof}
If $A_j$ consists of a single entry then $A_j$ and $\Vec{v}_j$ both equal $1$ by definition of pivot.  The vector $B_j$ consists of all non-pivot entries of $g_{j'}$. We will solve Equation~\eqref{equation: conditions on entries} and then analyze the tori ${\bf s}_j$ and ${\bf s'}_j$ from Theorem~\ref{theorem: main column test} more completely. When we restrict Equation~\eqref{equation: conditions on entries} to row $i$ then the equation simplifies to
\[{\bf s}_{ij} g_{ij'} = {\bf s'}_{ij} g_{ij'} \]
by Theorem~\ref{theorem: main column test}.  By hypothesis $g_{ij'} \neq 0$ so we may cancel to obtain ${\bf s}_{ij} = {\bf s'}_{ij}$.  Thus the claim will be proven once we identify the tori ${\bf s}_j$ and ${\bf s'}_j$.  

Since $A_j$ is a $1 \times 1$ matrix the torus ${\bf s}_{j}$ is the rank-one torus that scales the pivot entry of $Xg_j$.  Identifing coefficients, we find ${\bf s}_j = \left(t_{X(piv(Xg_j))}/t_{piv(Xg_j)}\right)$ which is the righthand side of the claim.  By contrast the torus ${\bf s'}_j$ has elements $t_{X(i)}/t_i$ for each row $i$ as before.  Thus we obtain the equations
\[\frac{t_{X(piv(Xg_j))}}{t_{piv(Xg_j)}} = \frac{t_{X(i)}}{t_i}\]
imposing conditions on $T$ for each nonzero entry of $g_j$ other than the pivot.
\end{proof}

\begin{corollary}\label{corollary:B nonzero}
We use the notation of Theorem~\ref{theorem: main column test}.  Suppose $X$ is skeletal nilpotent. If there exists $gB\in \Hess(X,H)$ and a column $j$ such that $B_j$ is a nonzero matrix, then $\Hess(X,H)$ is not $T$-stable. 
\end{corollary}

\begin{proof}
Suppose $gB\in \Hess(X,H)$ and $j$ is a column of $g$ such that $B_j$ is a nonzero matrix. Then Equation~\eqref{equation: conditions on entries} provides at least one nonzero equation among the entries of the vectors, since  $B_j$ is nonzero. We claim the equation restricts the possibilities for ${\bf t}$. Let $A_j^{-1} = (a_{pq})$. Then $A_j^{-1}{\bf s}_j\Vec{v}_j = \sum_q a_{pq} \left({\bf s}_j\right)_q v_{qj}$. Then if we write $B_j = (b_{rp})$, the left hand side of \eqref{equation: conditions on entries} is
\[
B_j A_j^{-1}{\bf s}_j\Vec{v}_j = \sum_{p,q} b_{rp}a_{pq}  \left({\bf s}_j\right)_q v_{qj},
\]
where the sum is over $p, q\in \mathcal R_j.$
On the other hand, the right hand side of \eqref{equation: conditions on entries} is the product 
\[
{\bf s}_j' B_j A_j^{-1}\Vec{v}_j = \sum_{p, q}  \left({\bf s'}_j\right)_r b_{rp}a_{pq} v_{qj},
\]
where again the sum is over elements $p, q\in \mathcal R_j.$
Taking a nonzero row $r$ of the resulting product results in the equation 
\[
\sum_{p,q} b_{rp}a_{pq}  \left({\bf s}_j\right)_q v_{qj} = \sum_{p, q}  \left({\bf s'}_j\right)_r b_{rp}a_{pq} v_{qj}.
\]
Note that $r$ indexes rows of $B_j$, so $r\not\in \mathcal R_j$. Since $r\not \in \mathcal R_j$, we may choose ${\bf t}$ so that $ \left({\bf s'}_j\right)_r \neq 1$ and $ \left({\bf s}_j\right)_q=1$ for all $q\in \mathcal R_j$. Then
\[
\sum_{p,q} b_{rp}a_{pq} v_{qj} =  \left({\bf s'}_j\right)_r \sum_{p, q}  b_{rp}a_{pq} v_{qj},
\]
which is obviously false. Since \eqref{equation: conditions on entries} fails to hold, Theorem~\ref{theorem: main column test} implies $\Hess(X,H)$ is not $T$-invariant.
\end{proof}

\begin{corollary} \label{corollary: trivial conditions to satisfy}
Use the notation of Theorem~\ref{theorem: main column test}. Let $gB\in \Hess(X,H)$. If any of the following hold for some column $j$ of $g$, then the system in Equation~\eqref{equation: conditions on entries} is vacuously satisfied:
\begin{itemize}
    \item The vector $Xg_j$ is zero.
    \item The matrix $B_j$ is zero.
    \item The matrix $B_j$ is empty.
    \item The matrix $A_j$ is a square $piv(Xg_j) \times piv(Xg_j)$ matrix.
\end{itemize}
If at least one of these conditions holds for every column $j$ of $g$, then ${\bf t}g\in \Hess(X,H)$ for all ${\bf t}\in T$. In particular, if for all $gB \in \Hess(X,H)$ and all columns $j$ of $g$, at least one of those conditions holds, then $\Hess(X,H)$ is $T$-stable.
\end{corollary}

\begin{proof}
If either of the first two conditions holds then Equation~\eqref{equation: conditions on entries} simply states that zero equals zero.  By construction of $A_j$ and $B_j$ the last two conditions are equivalent.  If they hold then Equation~\eqref{equation: conditions on entries} represents a system of zero equations and is trivially true. If at least one of these conditions holds for all $g$ and $j$ then Equation~\eqref{equation: conditions on entries} is satisfied for all ${\bf t} \in T$ and so all of $T$ acts on $\Hess(X,H)$.
\end{proof}

 Note that Example~\ref{example: matrices A and B in Hessenberg conditions} showed instances of several of the special cases in Corollary~\ref{corollary: trivial conditions to satisfy}.

The next few results present contexts in which one of these conditions holds, typically because of some degeneracy in the linear system.  While we do not discuss cases when the matrix $B_j$ is zero immediately, there are several natural ways for it to happen.  The easiest is if the pivot rows of $A_j$ are actually the pivot rows of the corresponding columns in $g$ without any additional permutation of rows.  In addition, the Hessenberg conditions can force some entries of row $i$ to be zero if which the basis vector $e_i$ is not in the image of $X$.

\begin{corollary} \label{corollary: last row of H empty}
Suppose that $X = F_{n-1}$, namely a regular nilpotent nilpotent matrix in Jordan form. If $h(1)=h(n-1)=n-1$ and $h(n)=n$ then $\Hess(X,H)$ is a $T$-stable subvariety of $G/B$. Moreover $\Hess(X,H)$ is a union of Schubert varieties.
\end{corollary}

\begin{proof}
Suppose $gB\in \Hess(X, H)$ and $g$ is in normalized Schubert form. First note that if $h(1)=h(n-1)=n-1$ then 
\[ Xg_j \in \langle g_1, g_2, \ldots, g_{n-1} \rangle= V_g^{n-1} \]
for all $j \leq n-1$.  It follows that for each column $g_j$ all successive images are contained in $V^{n-1}_g$
\[Xg_j, X^2 g_j, X^3 g_j, \ldots \in \langle g_1, \ldots, g_{n-1} \rangle.\]   
Since $X=F_{n-1}$ the pivot of $Xg_j$ is in the row above the pivot of $g_j$ unless $g_j = e_1$ in which case $Xg_j$ is zero.  It follows that if $g_j$ has pivot in row $k$ then $g_j$ and its successive images under $X$ span the first $k$ basis vectors:
\[ \langle g_j, Xg_j, X^2g_j, \ldots, X^{k-1} g_j \rangle = \langle e_k, e_{k-1}, \ldots, e_1 \rangle.\] 
Now suppose $g_j$ has pivot in row $n$ for some $j<n$. Together these two claims imply that 
\[\langle g_j, Xg_j, \ldots, X^{n-1}g_j \rangle = \mathbb{C}^n \subseteq V_g^{n-1},\]
which is a contradiction.  Thus $g_n$ must have a pivot in the $n$th row. Since $g$ is in normalized Schubert form, $g_n=e_n$. 

 Moreover, if $g$ is in normalized Schubert form, and $g_n=e_n,$ then 
$$\langle g_1, \dots g_{n-1}\rangle = \langle e_1, \dots, e_{n-1}\rangle$$ and thus the Hessenberg conditions are necessarily satisfied. In particular, if $g'B$ is in the same Schubert cell as $gB,$ then it has pivots in the same position and is also in $\Hess(X,H)$. Thus $\mathcal C_w\subseteq \Hess(X,H)$ or it intersects $\Hess(X,H)$ trivially. If $\Hess(X,H)$ contains $\mathcal C_w$, it contains the closure $X_w$ since $\Hess(X,H)$ is closed. Thus $\Hess(X,H)$ is a union of Schubert varieties; indeed, it is homeomorphic to flags on $\mathbb C^{n-1}.$
\end{proof}

 The previous result relied on the fact that only one basis vector is not in the image of $X$ and so cannot be easily extended to different nilpotent matrices.  The next result shows how different the underlying Hessenberg varieties can be, even with very similar Hessenberg space.  In this case, we use a kind of {\em dual} Hessenberg space where we omit the first column of $n \times n$ matrices instead of the last row.

\begin{corollary} \label{corollary: first column of Xg empty}
Suppose that $X$ is any skeletal nilpotent matrix.  If $h(1)=1$ and $h(2)=n$ then $\Hess(X,H)$ is $T$-stable.
\end{corollary}

\begin{proof}
In this case, for $gB \in \Hess(X,H)$ we must have 
\[Xg_1 \in \{cg_1: c \in \mathbb{C}\}.\]  
Since the pivot of $Xg_1$ and $g_1$ are in different rows when $X$ is nilpotent, we conclude that in fact $c=0$ and $g_1 \in \ker X$.  For every column $j$ with $Xg_j$ nonzero, the matrix $A_j$ is full rank because every column of $g$ is contained in  the coefficient matrix from the linear system of Equation \eqref{eq: Hessenberg conditions}. 
Thus the conditions of Corollary~\ref{corollary: trivial conditions to satisfy} hold for all $g$ with $gB \in \Hess(X,H)$ and all $j$.  We conclude that the full torus acts on $\Hess(X,H)$ as desired.
\end{proof}

Hessenberg varieties satisfying $h(1)=1$ and $h(2)=n$ form a family of fiber bundles with fiber a smaller flag variety.  Indeed, suppose $X$ is any nilpotent matrix and suppose its kernel is $m$-dimensional.  Then $g\in \Hess(X,H)$ if and only if  $V_g^1$ is a line in $\ker X$ and the next $n-1$ columns form a full flag in $\mathbb{C}^n/V_g^1$. Thus $\Hess(X,H)$ is a bundle over $\mathbb P^{m-1}$ whose fiber is the set of full flags in $\mathbb C^n/V_g^1$.
In particular, these Hessenberg varieties are smooth of dimension $\binom{n-1}{2} \times (m-1)$.

By contrast, when $h(1)=2$ we can identify cases in which the largest subtorus of $T$ under which $\Hess(X,H)$ is invariant is codimension $(k-1)$.  This is what the next corollary does; it is a key part of the proof of Theorem~\ref{theorem: max subtorus on Fk} and a core argument in the classical analysis of Peterson varieties.

\begin{corollary}\label{corollary: first column condition for Fk}
Fix $X$ to be a skeletal nilpotent with $k$ nonzero rows and let $H$ be any Hessenberg space with $h(1)=2$. Suppose there is an element $gB \in \Hess(X,H)$ with $piv(g_1)=n$ and $n$ nonzero entries in $g_1$. Then the largest subtorus of $T$ that acts on $\Hess(X,H)$ is the codimension-$(k-1)$ torus $K \subseteq T$ from Theorem~\ref{thm:X action by codim k-1 torus}. 
\end{corollary}

\begin{proof}
By construction of $X$ we know that $X(g_1)$ is nonzero with pivot in row at most $n-1$.  Thus $X(g_1)$ and $g_1$ together span a two-dimensional space.  Since $gB \in \Hess(X,H)$ and $h(1)=2$ we conclude $V_g^2 = \langle g_1, X(g_1) \rangle$. Thus $A_1$ is a $1\times 1$ matrix.

By construction of $X$ and the fact that each non-pivot entry of $g_1$ is nonzero we know that there are $k$ nonzero entries in $X(g_1)$.  Since each non-pivot entry of $g_1$ is nonzero, Corollary~\ref{corollary: one column pivot conditions} tells us that
\[\frac{t_{X(i)}}{t_i} = \frac{t_{X(i')}}{t_{i'}}\]
for all pairs of nonzero entries $(i,X(i)), (i',X(i'))$ in $X$.  These are precisely the $k-1$ conditions of Equation~\eqref{torusconstraintX} from Theorem~\ref{thm:X action by codim k-1 torus}.  So the codimension-$(k-1)$ torus $K \subseteq T$ is the maximal subtorus of $T$ that acts on these Hessenberg varieties, as claimed.
\end{proof}

\subsection{Torus actions on $\Hess(F_k,H)$} \label{section: torus actions for Fk}

We now show exactly when $\Hess(F_2,H)$ is $T$-stable.
Theorem~\ref{theorem: conditions for T-stability} is equivalent to the statement that $\Hess(F_2, H)$ is {\it not} $T$-stable exactly when $h$ is observed in Table~\ref{table: Hessenbergfunction}.  All of the Hessenberg spaces in Example~\ref{example: hessenberg spaces} are cases from Table~\ref{table: Hessenbergfunction}.

\begin{theorem}\label{theorem: conditions for T-stability} Suppose $n\geq 4.$ The variety $\Hess(F_2, H)$ is $T$-stable if and only if $h$ satisfies one of the following conditions:
\begin{enumerate}
    \item $h(j)=n$ for all $j=1,\dots, n$,
    \item $h(j) = n-1$ for $j=1, \dots, n-1,$ and $h(n)=n$,
    \item $h(1)=1$ and $h(j)=n$ for $j> 1$, or
    \item $h(1)=1$ and $h(j)=n-1$ for $j=2, \dots, n-1,$ and $h(n)=n$.
\end{enumerate}
\end{theorem}

\begin{table}[h]
\begin{tabular}{|ll|l|}
\hline
\multicolumn{2}{|l|}{Hessenberg function}                                    & Pivot Row \\ \hline
\multicolumn{1}{|l|}{\multirow{2}{*}{$h(n-1)=n$}}   & $h(1)=1, h(2)\leq n-1$ &  $(2,n,\ldots,n-1,1) $        \\ \cline{2-3} 
\multicolumn{1}{|l|}{}                              & $2\leq h(1)\leq n-1$   &      $(n,2,\ldots,n-1,1)$     \\ \hline
\multicolumn{1}{|l|}{\multirow{2}{*}{$h(n-1)=n-1$}} & $h(1)=1, h(2)\leq n-2$ &     $(2, n,\ldots,1,n-1)$      \\ \cline{2-3} 
\multicolumn{1}{|l|}{}                              & $2\leq h(1)\leq n-2$   & $(n,2,\ldots,1,n-1)$   \\ \hline
\end{tabular}

\caption{Hessenberg functions for which $\Hess(F_2, H)$ is not $T$-stable}\label{table: Hessenbergfunction}
\end{table}

\begin{proof}[Proof of Theorem~\ref{theorem: conditions for T-stability}]

Suppose $X=F_2$ and $gB\in \Hess(X,H)$. There are precisely two nonzero columns in $Xg$, namely
$$
    Xg_{\ell} =(a,1,0,0,\ldots,0)^{tr}, \quad \mbox{and}\quad
    Xg_{m}=(1,0,0,\ldots,0)^{tr}=e_1
$$ for some $a\in \mathbb C$. These are the images of the column vectors $g_{\ell}$ and $g_{m}$ of $g$ with pivot in the $n^{th}$ and $(n-1)^{th}$ rows, respectively. Observe that $\ell$ and $m$ have the same value for all $g'$ in the same Schubert cell as $g$. 

We first identify when Corollary~\ref{corollary: trivial conditions to satisfy} applies to all columns of $g$ and thus ${\bf t}gB\in\Hess(X,H)$ for all ${\bf t}\in T.$ Since $Xg_j=0$ for $j\neq m, \ell$, we need only consider $j=m$ and $j=\ell$. Furthermore, since $piv(Xg_m)=1$ and one of the first $h(m)$ columns of $g$ is $e_1$, $A_m$ is a $1\times 1$ matrix and $B_m$ is empty. Thus we need only find when Corollary~\ref{corollary: trivial conditions to satisfy} applies to $j=\ell$. By assumption $Xg_\ell\in V_g^{h(\ell)}$. Since $piv(Xg_\ell)=2$, the matrix  $A_\ell$ is either a $1\times 1$ or $2\times 2$ matrix. 

 If $Xg_m=e_1 \in V_g^{h(\ell)}$ then $A_\ell$ is a $2\times 2$ matrix, and $B_{\ell}$ is an empty matrix.  Corollary~\ref{corollary: trivial conditions to satisfy} applies to all such elements $gB\in \Hess(X,H)$, independent of $H$.

On the other hand, if $Xg_m=e_1\not\in V_g^{h(\ell)}$ for any $gB\in \Hess(X,H)$, then
$A_\ell$ is a $1\times 1$ matrix, and $B_\ell$ is the $1\times 1$ matrix $(b)$ where $(b,1,0,0,\ldots,0)^{tr}$ is the column of $g$ with pivot in row $2$. If $b\neq 0$, Corollary~\ref{corollary:B nonzero} implies that $\Hess(X,H)$ is not $T$-stable. 

The column $e_1$ of $g$ is necessarily to the right of the column $(b,1,0,0,\ldots,0)^{tr}$ in $g$, so $b$ is not identically zero on the Schubert cell $\mathcal C_w$ of $g$.
 Observe that the pivot positions of $g'$ are the same for all elements of $g'B\in \mathcal C_w$. Thus any $g'B \in \mathcal C_w$
  satisfies $Xg'_m=e_1\not\in V_{g'}^{h(\ell)}$. Then Equation~\eqref{eq: Hessenberg conditions} holds if and only if $g'$ in Schubert normal form satisfies Equation~\eqref{eq: matrix Hessenberg conditions}. We may choose $g'$ so that the column with pivot in the second row is $(b', 1, 0,\dots, 1)^{tr}$ with $b'\neq 0$. On the other hand, $Xg'_\ell= (a', 1,0, \dots, 0)^{tr}$ for some $a'\in\mathbb C.$ Then $g'B\in \Hess(X,H)$ if and only if \eqref{eq: matrix Hessenberg conditions} is satisfied, which is exactly when   $b'=a'$. Again by Corollary~\ref{corollary:B nonzero}, $\Hess(X,H)$ is not $T$-stable.

In sum, $\Hess(X,H)$ admits a full $T$-action if and only if $e_1\in V_g^{h(\ell)}$ for all $gB\in \Hess(X,H)$.  This in turn is determined completely by the shape of $H$.
Furthermore, if   $e_1\in V_g^{h(\ell)}$ holds for some $g$, then it holds for all elements in the same Schubert cell as $g$. Thus  $\Hess(X,H)$ admits a full $T$-action if and only if $e_1\in V_w^{h(\ell)}$ for every $wB\in \Hess(X,H)$ with $w\in W$. We therefore restrict our attention to the permutation matrices.

Observe that
the permutation matrix $w$ with $gB \in \mathcal C_w$ satisfies  $w(\ell)=n$ or ${we_\ell=e_n}$. 
The condition that $\ell<m$ is that $e_n$ occurs in an earlier column than  $e_{n-1}$ in the matrix $w$. Similarly the column $e_2$ occurs to the left of $e_1$ in $w$. 
This leaves only four possible orders for the vectors $e_1, e_2, e_{n-1},e_n$ to occur as columns of $w$, as listed in Column 2 of the chart below. 
In each case, the Hessenberg conditions may put a constraint on $h$, which is listed in the third column. We only specify four columns of each permutation matrix $w_i$ and assume that $w_i$ restricts to the identity on $e_j$ with $3 \leq j \leq n-2$.

Our strategy for each $H$ is to pick a specific permutation $w$ among $w_1, w_2, w_3$ or $w_4$ listed below and verify  that $wB\in \Hess(F_2, H)$ and $e_1\not\in V_w^{h(\ell)}.$ 
This will demonstrate that $\Hess(F_2,H)$ is not $T$-stable.

\medskip

\begin{tabular}{c|c|c|c|c}
 Permutation & Image on &  & $H$ with     &  Condition for  \\
matrix $w$ & $e_1, e_2, e_{n-1}, e_n$ & $\ell$ & $wB \in \Hess(F_2,H)$ &$e_1 \not \in V_w^{h(\ell)}$\\
\cline{1-5} $w_1$     & $e_2, e_n, e_1, e_{n-1} $ & 2 & All & $h(2) < n-1$ \\
$w_2$ & $e_n, e_2, e_1, e_{n-1}$ & 1 &  $h(1) \geq 2$ & $h(1) < n-1$ \\
$w_3$ & $e_2, e_n, e_{n-1}, e_1$ & 2 & $h(n-1) =n $ & $h(2)<n$ \\
$w_4$ & $e_n, e_2, e_{n-1}, e_1$ & 1 & $h(1) \geq 2$ and $h(n-1)=n$ & $h(1) < n$
\end{tabular}

\medskip

For example, if $w_1$ has column vectors $e_2, e_n, e_1, e_{n-1}$ in that order as its first two and last two columns, then $w_1B\in \Hess(F_2, H)$ for any vector subspace $H \subseteq \mathfrak g$. Note that $\ell=w^{-1}(n)$ so $\ell=2$. If $h(2)<n-1$ then $V_w^{h(2)}$ does not contain the last two columns of $w$ and so in particular does not contain $e_1.$  

The Hessenberg functions that remain to be considered are those for which $\Hess(F_2,H)$ does not contain $w_iB$ or for which $w_iB \in \Hess(F_2,H)$ and $e_1 \in V_w^{h(\ell)}$.  In other words:
\begin{itemize}
    \item those with $h(2) \geq n-1$, 
    \item those with $h(1) =1$, 
    or $h(1) \geq n-1$,
    \item those with $h(n-1)=n-1$ 
    or $h(2)=n$
    \item those with $h(1)=1$ or $h(n-1)=n-1$ or $h(1)=n$.
\end{itemize}
If all of those four conditions are true then $\Hess(F_2,H)$ is potentially $T$-stable. These conditions together on $h$ are equivalent to the four cases in statement of the theorem.  Our last task is to confirm that they are all in fact $T$-stable Hessenberg varieties.

\begin{enumerate}
\item If $h(1)=n$ then the Hessenberg variety is the entire flag variety so the full torus acts.

\item If $h(1)=n-1$ and $h(n-1)=n-1$ then Corollary~\ref{corollary: last row of H empty} applies and shows that $\Hess(F_2,h)$ admits the full torus action.

\item If $h(1)=1$ and $h(2)=n$ then Corollary~\ref{corollary: first column of Xg empty} applies and shows that $\Hess(F_2,h)$ admits the full torus action.

\item Finally suppose $h(1)=1, h(2)=n-1,$ and $h(n-1)=n-1$ and suppose $gB \in \Hess(F_2,H)$.  We know that $g_1$ has no pivot in row $n$ since $Xg_1 \neq g_1$ and thus $Xg_1$ cannot be in $V_g^1 = \langle g_1 \rangle$ unless it is zero. (See also the proof of Corollary~\ref{corollary: first column of Xg empty}.) So the column $g_{\ell}$ with pivot in row $n$ is in the last $n-1$ columns.  Similarly we know that $g_n \neq e_1$ since otherwise $Xg_m \not \in V_g^{n-1}$.  Thus $e_1 \in V_g^{h(\ell)}$.  It follows that $\Hess(F_2,H)$ admits the full torus action in this case, as well.
\end{enumerate}
This proves the claim.
\end{proof}

The next result finds a matrix $g$ satisfying the hypotheses of Corollary \ref{corollary: first column condition for Fk} for a family of Hessenberg varieties for $F_k$. One key consequence of the claim is that the only nontrivial subtorus of $T$ that acts on the well-known Peterson variety is the one-dimensional torus $S$ from Theorem~\ref{thm:X action by codim k-1 torus}.  Note the similarity in proof to that of Lemma~\ref{lemma: circle group for skeletal}.  That is because the matrix $g$ is in the Schubert cell $\mathcal{C}_w$ for the permutation $w$ defined in Lemma~\ref{lemma: circle group for skeletal}. 

\begin{theorem} \label{theorem: max subtorus on Fk}
The codimension-$(k-1)$ torus $K \subseteq T$ from Theorem~\ref{thm:X action by codim k-1 torus} is the maximal torus that acts on the following:
\begin{itemize}
    \item The Hessenberg variety $\Hess(F_k,H_0)$ where $H_0$ is defined by $h(i)=i+1$ for $1 \leq i \leq n-1$.
    \item The Peterson variety $\Hess(F_{n-1},H_0)$ where 
    $F_{n-1}$ is the nilpotent matrix in Jordan form with a single Jordan block.
\end{itemize}
\end{theorem}

\begin{proof}
Fix $H_0$ to be the subspace defined by $h(i)=i+1$ for all $i$. We construct a matrix $g$ satisfying the conditions of Corollary~\ref{corollary: first column condition for Fk} for each case. We proceed by induction on the number of columns in our matrix, confirming for each $j\leq k$ that the Hessenberg conditions are satisfied. 

The base case consists of the first column of $g$. Choose any nonzero entries $g_{i1}$ for the first column $g_1 = (g_{11}, g_{21}, \ldots, g_{n-1,1}, 1)^{tr}$.  Choose $g_2 = F_kg_1$.  By construction the matrix $\left(g_1 g_2 \right)$ satisfies the Hessenberg conditions.  

We repeat a similar process in general: 
\begin{enumerate}
\item For each $j$ with $F_k g_j \neq {\bf 0}$ we define $g_{j+1} = F_k g_j$.  The Hessenberg conditions are satisfied for $j$, because  $V_g^{h(j)} =V_g^{j+1}$ which clearly contains $g_{j+1}$.
\item For each $j$ with  $F_k g_j ={\bf 0}$ we define $g_{j+1}$ to be the standard basis vector $e_i$ with the largest possible pivot not represented among $\{g_1, g_2, \ldots, g_j\},$ in other words
\[ i = \max \{piv(g_1), piv(g_2), \ldots, piv(g_j)\}^c, \]
where the complement of the set is taken in $\{1,2,\ldots,n\}$.  The Hessenberg conditions are trivially satisfied for $j$, as ${\bf 0}\in V_g^{h(j)}.$ 
\end{enumerate}

We check that the matrix $g$ obtained in this fashion is invertible  by showing that the pivots of the $n$ columns are all distinct.  By construction the pivots of columns $g_1, g_2, g_3, \ldots$ are in rows 
\[n, n-(n-k), n-2(n-k), \ldots\] 
respectively.  In fact, whenever Step (1) is used to generate $g_{j+1}$ from $g_j$ the pivot rows of $g_{j+1}$ and $g_j$ are in the same congruence class modulo $n-k$.  Successive iterations of Step (1) generate pivots in rows 
\[piv(g_{j+1}), piv(g_{j+1})-(n-k), piv(g_{j+1})-2(n-k), \ldots\] 
until reaching the minimal nonnegative representative of the congruence class.  Step (2) is then iterated to add an $e_i$ with $k < i < n$ since those are the maximal $i \not \equiv n \mod (n-k)$. Step (2) ensures that every congruence class of $n-k$ is represented and Step (1) ensures that all nonnegative representatives of that congruence class in $\{1, 2, \ldots, n\}$ are represented.  So $g$ has pivots in each row.

Thus $g$ satisfies the hypotheses of Corollary~\ref{corollary: first column condition for Fk} proving the claim for $\Hess(F_k,H_0)$.  

If $H_0 \subseteq H'$ then $\Hess(F_k,H_0) \subseteq \Hess(F_k,H')$ by comparing Hessenberg conditions.  So $g$ also satisfies the hypotheses of Corollary~\ref{corollary: first column condition for Fk} for $\Hess(F_k,H')$.  Finally, the Peterson variety is the special case $\Hess(F_{n-1},H_0)$ so it, too, satisfies the hypotheses of Corollary~\ref{corollary: first column condition for Fk}. This proves the claim.
\end{proof}

\section{GKM Hessenberg varieties} \label{section: gkm spaces}

In this section we characterize many Hessenberg varieties that are GKM with respect to various tori.  Our first result is the observation that all the Hessenberg varieties that we proved admit the full $T$-action are GKM with respect to $T$.  In Theorem~\ref{theorem: F2 gkm conditions} we prove that there are GKM Hessenberg varieties with respect to other (smaller) tori as well.  We characterize all $\Hess(F_2,H)$ that are GKM with respect to a proper subtorus of $T$.  Intuitively, if $X$ is a linear operator for which $\Hess(X,H)$ admits the action of a subtorus $K \subseteq T$ for all $H$, we expect smaller $H$ to make $\Hess(X,H)$ GKM.  We use this to find the maximal $H$ with $\Hess(F_2,H)$ GKM.  Finally, we establish that several of the families of GKM Hessenberg varieties we identified are not Schubert varieties.  This resolves in the negative an open question about whether all GKM subvarieties of $G/B$ are unions of Schubert varieties.

\begin{theorem}
The following Hessenberg varieties $\Hess(X,H)$ are all GKM with respect to $T$:
\begin{itemize}
\item $X$ is skeletal nilpotent, and $h(j)=n-1$ for $j<n$.
\item $X$ is skeletal nilpotent, $h(1)=1$ and $h(2)=n$. 
\item $X=F_2$ and $h$ satisfies conditions (1), (2), (3), or (4) in Theorem~\ref{theorem: conditions for T-stability}.
\end{itemize}
\end{theorem}

\begin{proof}
We proved these were $T$-stable in Corollary~\ref{corollary: last row of H empty}, Corollary~\ref{corollary: first column of Xg empty}, and Theorem~\ref{theorem: conditions for T-stability}, respectively.  By Lemma~\ref{lemma: subvariety of GKM with same torus action is GKM} they are thus GKM with respect to $T$.
\end{proof}

Some of the Hessenberg varieties without the full torus action are nonetheless GKM with respect to a smaller subtorus of $T$. We use the tori $K$ from Theorem~\ref{thm:X action by codim k-1 torus} in what follows.  The next lemma proves that Hessenberg varieties satisfy the first condition of a GKM space with respect to this $K$-action, namely their fixed points are isolated.

\begin{lemma}\label{lemma:K fixed points isolated X skeletal}
Let $X$ be skeletal nilpotent and $K$ be the codimension-(k-1) torus specified in Theorem~\ref{thm:X action by codim k-1 torus}. The fixed point set $\Hess(X,H)^K$ is finite.
\end{lemma}

\begin{proof}
Lemma~\ref{lemma: circle group for skeletal} guarantees a permutation $w$ such that $wSw^{-1}$ acts on $\Hess(X,H)$.  Every ${\bf t}\in wSw^{-1}$ is given by $(t^{w(1)}, t^{w(2)}, \dots, t^{w(n)})$ for some constant $t \in \mathbb{C}^*$ and furthermore $w$ is constructed so that $t^{-w(i)}t^{w(X(i)}=t$ whenever $X$ is nonzero in row $i$. Any such ${\bf t}$ also satisfies the equations $\eqref{torusconstraintX}$ for all nonzero rows $i$ of $X$ by construction. It follows that $wSw^{-1}\subseteq K$.
By Corollary~\ref{corollary: circle isolated fixed points}, the $wSw^{-1}$-fixed points in $\Hess(X,H)$ are finite and thus so are the $K$-fixed points.
\end{proof}

The following theorem describes a family of Hessenberg varieties for $F_2$ that are GKM with respect to the codimension-one torus $K$ from Theorem~\ref{thm:X action by codim k-1 torus} but do not admit a full $T$-action.  Moreover, we will show that no other Hessenberg varieties for $F_2$ are GKM with respect to $K$.

\begin{theorem} \label{theorem: F2 gkm conditions}
Fix $n \geq 4$. Let $K$ be the codimension-1 torus that acts on $\Hess(F_2,H)$ as specified in Theorem~\ref{thm:X action by codim k-1 torus}. 
 Then $\Hess(F_2,H)$ is GKM with respect to $K$ if and only if there exists some $i \in \{1,\dots, n-2\}$ so that  
\[ \begin{array}{rl} h(i) &\leq i+2, \\ h(i+1) &\leq i+2, \end{array}\]
and otherwise $h(j)=j$.  
\end{theorem}

\begin{proof}
Lemma~\ref{lemma:K fixed points isolated X skeletal} showed that the fixed point set $\Hess(F_2,H)^K$ is finite so we only need to prove that $\Hess(F_2,H)$ has a finite number of one-dimensional $K$-orbits.

First we show that $\Hess(F_2,H)$ has a finite number of one-dimensional $K$-orbits if and only if there is no $gB \in \Hess(F_2,H)$ with $g=uw$ where $u$ is upper-triangular with ones along the diagonal and nonzero in both positions $(2,n)$ and $(1,n-1)$.  Indeed, suppose $u$ is an upper-triangular matrix with ones along the diagonal.  Note that 
$$
{\bf t} uwB  = {\bf t} u {\bf t}^{-1}{\bf t} w {\bf t^{-1}}B = ({\bf t}^{-1} u {\bf t})wB,
$$
where the conjugate ${\bf t}^{-1}u{\bf t}$ is also upper-triangular with ones along the diagonal.  Thus the normalized Schubert form for the flag ${\bf t}uwB$ is ${\bf t}^{-1}u{\bf t}w$.  If $i<j$ then entry $(i,j)$ of ${\bf t}^{-1}u{\bf t}$ is 
\[ \frac{t_j}{t_i}u_{ij}. \]
Suppose there are two entries $\{(i,j), (i',j')\} \neq \{(2,n),(1,n-1)\}$ with $u_{ij}$ and $u_{i'j'}$ both nonzero.  Then the $K$-orbit of $u$ is at least two-dimensional since the $K$-orbit agrees with the $T$-orbit in these two coordinates.  This means that the one-dimensional orbits of $K$ occur either when exactly one entry $u_{ij}$ is nonzero and other entries off the diagonal vanish, or possibly when the two entries $u_{2,n}$ $u_{1,n-1}$ are both nonzero and other entries off diagonal vanish. 

If $u$ has one nonzero entry $u_{ij}$, then different values of $u_{ij}$ are just multiples of one another and so lie in the same $K$-orbit. Thus there is at most one orbit for each entry of $u$ and so a finite number of these one-dimensional orbits. 

Now suppose that $u_{2,n}$ and $u_{1,n-1}$ are both nonzero. In this case the $K$-action on these coordinates is given by
\[ \left( \frac{t_n}{t_2} u_{2,n}, \frac{t_{n-1}}{t_1} u_{1,n-1}  \right), \]
which is a one-dimensional orbit since $t_n/t_2 = t_{n-1}/t_1$. 
Since each $uw$ is in normalized Schubert form, two orbits are distinct if the ratio $u_{2,n}/u_{1,n-1}$ is distinct. 
If there are an infinite number of distinct orbits, then
$\Hess(F_2,H_i)$ is not GKM. 

The image $F_2(uw)$ does not depend on  $u_{2,n}$ or $u_{1,n-1}$ by construction of $F_2$ and no conditions can arise on $u_{2,n}$ or $u_{1,n-1}$ in the linear system of Equation~\eqref{equation: conditions on entries}. 
So if any element of the Hessenberg Schubert cell for $w$ has both $u_{2,n}$ and $u_{1,n-1}$ nonzero then there are an infinite number of distinct ratios $u_{2,n}/u_{1,n-1}$ in the Hessenberg Schubert cell.

We have proven that $\Hess(F_2,H)$ is GKM with respect to $K$ if and only if at most one of entries $u_{2,n}, u_{1,n-1}$ is nonzero for any flag $uwB \in \Hess(F_2,H)$. We now identify a class of spaces $H$ for which $\Hess(F_2,H)$ is guaranteed not to be GKM. 

Suppose there are $j, j'$ so that $h$ satisfy one of the following two conditions:
\begin{itemize}
\item $h(j) \geq j+1$ and $h(j') \geq j'+1$, with $j'-j>1$, and
\item $h(j) \geq j+2$ and $h(j+1)\geq j+3$.
\end{itemize}
In either case, we find a permutation matrix $w$ that satisfies the Hessenberg conditions, and for which there exists $uwB\in \Hess(X,H)$ with nonzero entries $u_{2,n}, u_{1,n-1}$. 

The Hessenberg Schubert cell $\mathcal{C}_w \cap \Hess(X,H)$ contains an element $uwB$ with entries $u_{2,n}, u_{1,n-1}$ both nonzero if and only if
the column of $uw$ with pivot in row $2$ occurs after the column with pivot in row $n$ and similarly the column with pivot in row $1$ occurs after the column with pivot in row $n-1$. This is independent of $u$ and just a property of $w$, occurring when the following equations hold:
\begin{equation}\label{uw in cell}
w^{-1}(2)>j \quad \mbox{and}\quad w^{-1}(1)>j'.
\end{equation}

Suppose $h$ satisfies the first condition. 
Let $w$ be the permutation satisfying
$w(j)=n, w(j+1)=2, w(j')=n-1, w(j'+1)=1$ and with all other entries in increasing order. Whenever $h(j) \geq j+1$ and $h(j') \geq j'+1$ the vectors $F_2w_j$ and $F_2w_{j'}$ both satisfy the Hessenberg conditions. Thus $wB \in \Hess(F_2,H)$. Furthermore $w^{-1}(2)=j+1>j$ and $w^{-1}(1)=j'+1>1$ so $uwB\in \mathcal{C}_w \cap \Hess(F_2,H)$ includes those with $(u_{2,n}, u_{1,n-1})$ any pair of complex numbers. 

Suppose $h$ satisfies the second condition. 
Let $w$ be the permutation satisfying
$w(j)=n, w(j+1)=n-1, w(j+2)=2,w(j+3)=1$    and with all other entries in increasing order.   Whenever
$h(j) \geq j+2$ and $h(j+1)\geq j+3$ the vectors $F_2w_{j}$ and $F_2w_{j+1}$ both satisfy the Hessenberg conditions, implying $wB \in \Hess(F_2,H)$.  Furthermore $w^{-1}(2) =j+2 > j$ and $w^{-1}(1) = j+3>(j+1)+1$. Thus both $u_{2,n}, u_{1,n-1}$ may be arbitrary nonzero entries among elements $uwB$ of the Hessenberg Schubert cell $\mathcal{C}_w \cap \Hess(F_2,H)$.

We have shown that if $h$ satisfies either of the two bullet points above for any $j$, then $\Hess(F_2,H)$ is not GKM.

On the other hand, the maximal $H_j$ that violates both conditions satisfies $h(j)=h(j+1)=j+2$ for some $j$ and otherwise is the identity function.  We show that $\Hess(F_2,H)$ is a GKM space. We need to confirm that no permutation $w$ with $wB\in \Hess(F_2, H_j)$ has both $u_{2,n}, u_{1,n-1}$ nonzero in its Schubert cell.  If $u_{2,n}\neq 0$ then $e_2$ occurs to the right of $e_n$ in $w$. This can only occur if $e_n$ occurs in columns $j$ or $j+1$.  Similarly if $u_{1,n-1}\neq 0$ then $e_1$ occurs to the right of $e_{n-1}$ in $w$ so $e_{n-1}$ must be one of columns $j$ or $j+1$. But if both $e_n, e_{n-1}$ are in columns $j, j+1$ then we don't have enough columns to put $e_1, e_2$ while respecting the Hessenberg conditions without putting $e_1$ or $e_2$ in one of the first $j-1$ columns, and thus to the left of either $e_{n-1}$ or $e_n$. 
\end{proof}

\begin{remark}
Note that when $n \geq 6$ the Hessenberg varieties $\Hess(F_2,H)$ from the previous theorem do not admit a full torus action.  It follows that these Hessenberg varieties are GKM with respect to the codimension-one torus $K$ but {\it not} Schubert varieties. For more complete results, see recent work of Escobar, Precup, and Shareshian identifying which Hessenberg varieties are Schubert varieties \cite{EPS}.
\end{remark}

 Finally, we observe that some Hessenberg varieties with full torus actions may not be Schubert varieties.
 
 \begin{proposition} \label{proposition: Hess with full torus action}
 If $\Hess(X,H)$ admits an action of the full torus $T$ then the intersection of each Schubert cell $\mathcal{C}_w \cap \Hess(X,H)$ is an affine space (possibly empty) in which each coordinate is either zero or free. 
 \end{proposition}
 
\begin{proof}
Consider the equation $B_jA_j^{-1}{\bf s}_j\Vec{v}_j = {\bf s'}_j B_jA_j^{-1}\Vec{v}_j$ from Theorem~\ref{theorem: main column test}.  If the full torus $T$ acts on $\Hess(X,H)$ then this equation holds for each row $i$.  Restricting to $i$ gives 
\[\sum_{\ell} ({\bf s}_{j})_{\ell} c_{i\ell} = ({\bf s'}_j)_i d_i,\]
where $d_i$ and the $c_{i,\ell}$ are constants depending on $B_j, A_j^{-1}, \Vec{v}_j,$ and $\Vec{v}_j'$.  This is a linear equation on the torus entries unless the $c_{i \ell}, d_i$ are identically zero in any entries where ${\bf s}_j, {\bf s'}_j$ are not identically one.  The latter correspond to rows that are not in the image of $X$ and so impose no conditions on the entries of $g$.  This proves the claim.
\end{proof}

\begin{remark} \label{remark: non-Schubert GKM}
It is possible to find Hessenberg varieties that satisfy Proposition~\ref{proposition: Hess with full torus action} and that are not unions of Schubert varieties.  For instance, consider the case when $X$ is the subregular nilpotent given by $X = \sum_{i=2}^{n-1} E_{i,i+1}$ and when $H = \mathfrak{b}$ or equivalently $h(i)=i$ for all $i$. (Note that $X$ is not any of the $F_k$.) This is a Springer fiber and is a sequence of copies of $\mathbb{P}^1$ each joined successively at a point.  The cells for $n=4$ are given in Figure~\ref{figure: example of Hessenberg varieties with full torus action}. The reader can verify that several of these Hessenberg Schubert cells are proper subspaces of the corresponding Schubert cell in the full flag variety.
\begin{figure}[h]
\[
\left( \begin{array}{cccc} 1 & 0 & 0 & 0 \\ 0 & 1 & 0 & 0 \\ 0 & 0 & 1 & 0 \\ 0 & 0 & 0 & 1 \end{array} \right), 
\left( \begin{array}{cccc} a & 1 & 0 & 0 \\ 1 & 0 & 0 & 0 \\ 0 & 0 & 1 & 0 \\ 0 & 0 & 0 & 1 \end{array} \right),
\left( \begin{array}{cccc} 0 & b & 1 & 0 \\ 1 & 0 & 0 & 0 \\ 0 & 1 & 0 & 0 \\ 0 & 0 & 0 & 1 \end{array} \right),
\left( \begin{array}{cccc} 0 & 0 & c & 1 \\ 1 & 0 & 0 & 0 \\ 0 & 1 & 0 & 0 \\ 0 & 0 & 1 & 0 \end{array} \right)
\]
\caption{Cells for the subregular Springer fiber when $n=4$} \label{figure: example of Hessenberg varieties with full torus action}
\end{figure}

Moreover, the Hessenberg variety shown in Figure~\ref{figure: example of Hessenberg varieties with full torus action} cannot be translated to a homeomorphic Hessenberg variety that is simply the union of the four Schubert cells of dimension at most one.  Indeed, suppose there exists an element $g \in \G$ so that $g\Hess(X,\mathfrak{b})$ consists of the four Schubert cells corresponding to permutations $id, (12), (23), (34)$.  As in the proof of Corollary~\ref{corollary: first column of Xg empty}, we know that 
\[Xg\Vec{v}_1 \in \langle g\Vec{v}_1 \rangle \textup{  if and only if  } \Vec{v}_1 \in g^{-1} \ker X.\]  
Since $g$ is invertible we conclude $g^{-1} \ker X$ is two-dimensional.  Knowing the Schubert cells $id$ and $(12)$ are both in $g \Hess(X,\mathfrak{b})$ implies that their first columns $e_1$ and $ae_1+e_2$ are both in $g^{-1} \ker X$.  Since $e_1$ and $ae_1+e_2$ are linearly independent
\[\langle e_1, ae_1+e_2 \rangle = \langle e_1, e_2 \rangle =  g^{-1} \ker X.\]
Examining the second column of the Schubert cell corresponding to $(23)$, we conclude $ae_2+e_3 \in X^{-1} \ker X$.  Since $X^{-1} \ker X$ is three-dimensional, we can actually conclude $X^{-1} \ker X = \langle e_1, e_2, e_3 \rangle$.  Finally, the first two columns of the Schubert cell corresponding to $(34)$ are $e_1$ and $e_2$ and the third column is $ae_3+e_4$.  But we just showed that $X(ae_3+e_4) \not \in \ker X$ which is the span of the first two columns.

This argument can be extended to conclude that subregular Springer fibers for $n \geq 4$ are not unions of the $0$- and $1$-dimensional Schubert varieties despite both having the same Betti numbers.
\end{remark}

\section{Questions} \label{section: questions}

This paper only begins the systematic study of torus actions and GKM theory for Hessenberg varieties.  We end by posing open questions raised by this work.

\subsection{Torus actions and GKM spaces}

The first set of questions are immediate extensions of the results in this paper.

\begin{question}
Identify conditions on $H$ so that $\Hess(F_k,H)$ admits the action of a rank-$m$ subtorus of $T$.  Can the conditions be generalized to other nilpotent operators $X$?
\end{question}

\begin{question}
Suppose $X$ is nilpotent and $\Hess(X,H)$ has the action of a codimension-$(k-1)$ subtorus $K \subseteq T$.  When is $\Hess(X,H)$ GKM with respect to $K$?  
\end{question}

These questions could be more illuminating---and still valuable to researchers---for special families of $H$ or $X$.  We highlight three families  that arise frequently in research.

\begin{question}
If $H = \mathfrak{b}$ can we identify the largest torus action for the following important families of $\Hess(X,H)$, and are these Hessenberg varieties GKM for that action?  What can we say when $h(i)=i+1$ for all $i<n$?  What can we say when $X$ consists of $m$ equal-sized Jordan blocks?
\end{question}

Since Schubert varieties are $T$-stable, the following special case is also of interest, especially given recent work of Escobar, Precup, and Shareshian that determined which Hessenberg varieties in type $A_{n-1}$ are Schubert varieties \cite{EPS}.

\begin{question}
Which $\Hess(X,H)$ are unions of Schubert varieties?
\end{question}

\subsection{Combinatorial descriptions of $\Hess(X,H)$ that are GKM or $K$-stable}
Theorems~\ref{theorem: conditions for T-stability} and~\ref{theorem: F2 gkm conditions} identified the $T$-stable $\Hess(F_2,H)$ and GKM spaces $\Hess(F_2,H)$.  Their proofs identified a kind of permutation pattern that, if present in a flag in $\Hess(F_2,H)$, determined what kind of torus action applied to $\Hess(F_2,H)$. The conditions we gave may be more fruitfully stated in terms of inclusions of Hessenberg spaces or root subsets, or pattern avoidance within certain permutations, or using some other combinatorial construction.

More generally, we ask if there are always combinatorial conditions that characterize whether $\Hess(X,H)$ admits the action of a particular torus or is GKM.  

\begin{question}
Can the conditions for $\Hess(X,H)$ to have a particular torus action always be stated in terms of permutation patterns, shapes of Young tableaux, or other combinatorial conditions?  Can the conditions for $\Hess(X,H)$ to be GKM be given in combinatorial terms?
\end{question}

More specifically, while Corollaries~\ref{corollary:B nonzero} and~\ref{corollary: trivial conditions to satisfy} and Proposition~\ref{proposition: Hess with full torus action} gave some necessary and sufficient algebraic conditions for $\Hess(X,H)$ to have a full $T$-action, they do not provide an immediate test.  Thus we ask the following.

\begin{question}
Find combinatorial conditions on $X$ and $H$ that characterize all $T$-stable $\Hess(X,H)$.
\end{question}

\subsection{Hessenberg varieties in other Lie types}

This paper primarily addressed the case of Lie type $A_{n-1}$.  Some results extend to all Lie types; for instance, we generalize the proof that $F_1$ is $T$-stable in the next proposition.  This is a key result from \cite{AbeCro16}.

\begin{proposition}\label{proposition: F1 in all types}
Suppose $E_{\alpha} \in \mathfrak{g}$ and $H$ is a Hessenberg space.  Then $\Hess(E_{\alpha}, H)$ is $T$-stable and thus GKM with respect to $T$.
\end{proposition}

\begin{proof}
Suppose $E_{\alpha} \in \mathfrak{g}$ and ${\bf t} \in T$.  Then $Ad({\bf t})E_{\alpha} = \alpha({\bf t}) E_{\alpha}$ by definition of the adjoint action.  Since $\alpha({\bf t}) \in \mathbb{C}^*$ is a nonzero scalar, we have 
\[Ad(g^{-1}) (Ad({\bf t})E_{\alpha}) \in H \Longleftrightarrow \alpha({\bf t}) Ad (g^{-1})E_{\alpha} \in H \Longleftrightarrow Ad(g^{-1})E_{\alpha} \in H\]
By Lemma~\ref{lemma: subvariety of GKM with same torus action is GKM} the claim holds.
\end{proof}

However, the arguments in this paper, and especially in Theorem~\ref{theorem: main column test}, rely on specific properties of Lie type $A_{n-1}$ and the limited ways that roots can sum to form other roots.  We expect a result like Theorem~\ref{theorem: main column test} to be more complicated in other Lie types, if it could even be extended directly. So we ask:

\begin{question} \label{question: general Lie type}
Which of these results extend to general Lie type?
\end{question}

For example, we conjecture Corollaries~\ref{corollary: last row of H empty} and~\ref{corollary: first column of Xg empty} generalize to all Lie types.

\begin{question}
Suppose $\Delta$ is the set of simple roots for a Lie algebra $\mathfrak{g}$ of rank at least $4$.

\begin{itemize}
\item Let $H_n \subseteq \mathfrak{g}$ be the maximal parabolic subalgebra associated to the root system with simple roots $\Delta \backslash \alpha_n$.  Is $\Hess(X,H_n)$ a $T$-stable Hessenberg variety?
\item Let $H_1 \subseteq \mathfrak{g}$ be the maximal parabolic subalgebra associated to the root system with simple roots $\Delta \backslash \alpha_1$.  Is $\Hess(X,H_1)$ a $T$-stable Hessenberg variety? 
\end{itemize}
\end{question}

\subsection{Limiting behavior of Hessenberg spaces}

We can also consider the limiting behavior of $\Hess(X_n,H_n)$ for various sequences $\{(X_n, H_n)\}_{n \geq 1}$ of pairs of operators and Hessenberg spaces.  For instance, suppose that $X_n$ is the $n \times n$ matrix of form $F_2$.  Then Theorem~\ref{theorem: F2 gkm conditions} says:
\begin{itemize}
    \item When $H_n$ is given by the Hessenberg function $h(1)=h(3)=3$ and $h(i)=i$ for all other $i \leq n$ then the sequence $\{\Hess(X_n,H_n)\}$ stabilizes as a GKM space.
    \item By contrast, when $H_n$ is given by the Hessenberg function $h(i)=n-2$ for all $i \leq n-2$ and $h(i)=i$ otherwise, then the sequence $\{\Hess(X_n,H_n)\}$ stabilizes as a non-GKM space.
\end{itemize}

This is one way to quantify and make precise our intuition that ``most" Hessenberg varieties have certain behavior.  We obtain the following question.

\begin{question}
Suppose that $\{(X_n,H_n)\}_{n \geq 1}$ is a sequence of pairs of $n \times n$ nilpotent operators and Hessenberg spaces. Does the full torus eventually act on $\{\Hess(X_n,H_n)\}$? Which sequences $\{\Hess(X_n,H_n)\}$ stabilize as GKM spaces?  
\end{question}

\subsection{GKM theory of Hessenberg varieties}

Finally, while this paper has focused on the questions of {\it when} Hessenberg varieties have torus actions with respect to which they are GKM, the obvious follow-up question is: {\it what} is the equivariant cohomology of the GKM Hessenberg spaces? Abe and Crooks have answered this question for $\Hess(F_1,H)$ in all Lie types \cite{AbeCro16}.

\begin{question}
We have identified many Hessenberg varieties that are GKM spaces.  What is the equivariant cohomology of these $\Hess(X,H)$?  Is there a  combinatorial formulation along the lines of classical Schubert calculus, puzzles, and so on?
\end{question}

\bibliography{main}
\bibliographystyle{plain}

\end{document}